\documentclass[twoside,a4paper,11pt,leqno]{article}
\usepackage{latexsym,amsfonts,amssymb,amsmath,amscd}
\usepackage[latin1]{inputenc}
\usepackage{amsmath}
\usepackage{epsfig}
\usepackage{amsfonts}
\usepackage{multicol}
\usepackage{graphicx}
\def \R {I\!\!R}

\newenvironment{proof of lemma}
{\medskip\noindent{\bf Proof of Lemma 1.1\/}} {\null \hfill $\Box$
\par\medskip}

\newenvironment{proof}
{\medskip\noindent{\bf Proof.\/}} {\null \hfill $\Box$
\par\medskip}
\newcommand{\ba}{\begin{eqnarray}}
\newcommand{\ea}{\end{eqnarray}}

\newcommand{\basn}{\begin{eqnarray*}}
\newcommand{\easn}{\end{eqnarray*}}

\newtheorem{theorem}{\bf Theorem}[section]

\newtheorem{lemma}{\bf Lemma}[section]
\newtheorem{proposition}[theorem]{\bf Proposition}

\newtheorem{remark}{\bf Remark}[section]

\numberwithin{equation}{section}
\newcommand\AMSname{AMS subject classifications}

\title{A nonlinear general Neumann problem involving two critical exponents}
\author{Rejeb Hadiji\thanks{Universit\'e Paris-Est, Laboratoire d'Analyse et
de Math\'ematiques Appliqu\'ees (LAMA), CNRS : UMR 8050, 
61, Avenue du G\'en\'eral de Gaulle B\^at. P3, 4e
\'etage, F-94010 Cr\'eteil Cedex, France.\hskip3cm E-mail : hadiji@u-pec.fr}
$\,$ and Habib Yazidi
\thanks{E-mail : habib.yazidi@u-pec.fr}$^\dag$ $^\ast$ }
\date{}
\begin{document}
\maketitle

\begin{abstract}
We discuss the existence of solutions to the following nonlinear problem involving two critical Sobolev exponents
\begin{eqnarray*}\left\{
\begin{array}{lll} -\textsl{div}(p(x)\nabla u)=\beta |u|^{2^{*}-2}u+f(x,u) &\textrm{in $\Omega$,}\\
u\not\equiv0 &\textrm{in $\Omega$,}\\
\frac{\partial u}{\partial \nu}=Q(x)|u|^{2_{*}-2}u
&\textrm{on $\partial\Omega$,}
\end{array}
\right. \end{eqnarray*}
where $\beta \geq 0$, $Q$ is continuous on $\partial\Omega$, $p\in H^{1}(\Omega)$ is continuous and positive in
$\bar{\Omega}$ and $f$ is a lower-order perturbation of $|u|^{2^{*}-1}$ with $f(x,0)=0$.
\medskip

\noindent Keywords : {Sobolev critical exponent, The trace
embedding, Variational problem, Critical nonlinearity in the
boundary, Palais-Smale Condition, The mean curvature.}
\medskip

\par
\noindent2010 \AMSname: 35J20, 35J25, 35J60.
\end{abstract}
\section{ Introduction}
In this work, we deal with the following problem
\begin{eqnarray} \left\{
\begin{array}{lll} -\textsl{div}(p(x)\nabla u)=\beta |u|^{2^{*}-2}u+f(x,u) &\textrm{in $\Omega$,}\\
\hspace{21mm}u\not\equiv0 &\textrm{in $\Omega$,}\\
\hspace{20mm}\frac{\partial u}{\partial \nu}=Q(x)|u|^{2_{*}-2}u
&\textrm{on $\partial\Omega$,}
\end{array}
\right. \label{nouveau} \end{eqnarray} where $\Omega\subset \R^{N},\, N\geq
3$, is a bounded domain with the smooth boundary $\partial\Omega$,
$\nu$ is the outer normal on $\partial \Omega$, $\beta \geq 0$ is a constant, the coefficient $Q$ is continuous on $\partial\Omega$, the
coefficient $p\in H^{1}(\Omega)$ is continuous and positive in
$\bar{\Omega}$ and $f(x,u):\Omega\times
\R\rightarrow \R$ is measurable in $x$, continuous in $u$.\\
Here, $2_{*}=\frac{2(N-1)}{N-2}$ is the critical Sobolev exponent for the
trace embedding of the space $H^{1}(\Omega)$ into
$L^{2_{*}}(\partial\Omega)$ and $2^{*}=\frac{2N}{N-2}$ is the critical Sobolev exponent for the embedding $H^{1}(\Omega)$ into $L^{2^{*}}(\Omega)$. Both embedding are continuous, but not compact. Our goal is to study the existence of solutions to problem (\ref{nouveau}).\\
The main motivation to consider such problem is the study of conformal deformations of Riemannian manifolds with boundary, see \cite{C}, \cite{E1} and \cite{E2}.\\
 Problem(\ref{nouveau}) has a variational form. Then the eventual solutions correspond to the critical points of the energy functional.

 The existence of a solution of (\ref{nouveau}) is closely related to $S$ (resp. $S_{1}$) which is the best Sobolev constant for the imbedding $H^{1}(\Omega)$ into $L^{2^{*}}(\Omega)$ (resp. for the imbedding $H^{1}(\Omega)$ into
$L^{2_{*}}(\partial\Omega)$). As in \cite{BN} for the nonlinear
Dirichlet problem with critical Sobolev exponent,  we will fill out
the sufficient conditions to find  solutions for the problem in
presence of a nonlinear Neumann boundary data with a critical
nonlinearity. One of the difficulty of our problem, besides the fact
that the associated functional does not satisfy the Palais-Smale
compactness condition (PS), is that it possesses four levels of
homogeneity.

Let us recall some works related to the problem (\ref{nouveau}). If
$p\equiv 1$ and $u$ satisfies homogeneous Dirichlet condition,
problem (\ref{nouveau}) has been treated in \cite{BN}, where the
authors obtained positive solutions with energy less than
$\frac{1}{N} S^{\frac{N}{2}}$, see also \cite{H1} and \cite{CHL}. In \cite{S}, the author gives a
complete description of the energy levels $c$, associated to problem
(\ref{nouveau}), on which $(PS)_{c}$ sequence is not compact. For
the case $p\not\equiv 1$, $f(x,u)=\lambda \,u$ with homogeneous
Dirichlet condition we refer the reader to \cite{HY, HMPY} . For the
homogeneous Neumann problem, in \cite{CR}, the authors proved the
existence of solution with
energy less than $\frac{1}{2 N} S^{\frac{N}{2}}$.\\
The case $p\equiv
1\equiv Q$, $\beta= 0$ and $f(x,u)$ is a linear perturbation, has an
extensive literature and the first existence results was treated in
\cite{AM, AY, CW, CY}. In this case the solutions are obtained as
minimizers of the variational problem associated to (\ref{nouveau})
with energy less than $S_{1}$. If $\beta=0$ and $f(x,u)$ has an
explicit form, problem (\ref{nouveau}) has been studied in \cite{Y1,
Y2} and some existence results are obtained.

In \cite{CFT}, the authors were interested to the case $p\equiv 1$, $f(x,\,u)=0$ and the presence of two critical nonlinearities. They derived some existence results by the use of the concentration compactness principle see  \cite{L}. For another form of equation (\ref{nouveau}) with competing critical nonlinearities, see \cite{PT} and references therein.

In this paper we are concerned with the general case, more precisely, $p\not\equiv 1$, $Q \not\equiv 0$ and $f(x,u)\not = 0$. We assume that $f$ is a lower-order
perturbation of $|u|^{2^{*}-1}$ and $f(x,0)=0$.\\
Let $p_{0}=\displaystyle\min_{x\in\bar{\Omega}}p(x)$ and
$x_{0}\in\partial\Omega$ satisfy
$$\frac {(Q(x_{0}))^{N-2}}{p(x_{0})}=\max_{x\in\partial\Omega}\frac {|Q(x)|^{N-2}}{p(x)}.$$
We assume that
\begin{equation}|p(x)-p(x_{0})|=o(|x-x_{0}|)\label{eqalphap}\end{equation} and
\begin{equation}
|Q(x)-Q(x_{0})|=o(|x-x_{0}|)\label{eqalphaq}\end{equation}
for $x$ near $x_{0}.$\\

Our first contribution to problem (\ref{nouveau}), in section 2, is an existence result for the case where $\beta=0$. The energy solutions which we find are under the level on which the (PS) condition failed. More precisely, we show existence of solutions with energy in $]0,\, \frac{1}{2(N-1)}\frac{p(x_{0})}{(Q(x_{0}))^{N-2}}S_{1}^{N-1}[$.\\

Next, in section 3, we turn to the general case and look for solutions for problem (\ref{nouveau}) in the case of the presence of competing critical nonlinearities in the case $p(x_{0})=p_{0}$.\\
The main difficulty of the problem in caused by the presence of two
critical exponents and a general nonlinear perturbation. This fact
causes the change in energy level for which the Palais Smale
condition (PS) is not satisfied. In this paper, we determine
explicitly the new energy level $M(S,S_{1})$ defined by
\begin{equation}
M(S,S_{1})=\displaystyle \frac{\frac{p(x_{0})}{(Q(x_{0}))^{N-2}}S_{1}^{N-1}\,2^{N-2}}{\left[1+\sqrt{1+4 E}\right]^{N-2}}\displaystyle\left[\frac{1}{N}-\frac{N-2}{N(N-1)}\frac{1}{1+\sqrt{1+4 E}}\right]
\label{eqnouveau valeur1}
\end{equation}
where $E=\left(\frac{\frac{p(x_{0})}{Q(x_{0})^{N-2}}S_{1}^{N-1}}{(p_{0}S)^{\frac{N}{2}}}\right)^{\frac{2}{N-2}}$.
We will show the existence of solution for (\ref{nouveau}) with energy in $]0,\,M(S, S_{1})[$.\\
Note that
$$0<M(S,S_{1})< \min\left\{\frac{1}{2(N-1)}\frac{p(x_{0})}{(Q(x_{0}))^{N-2}}S_{1}^{N-1},\,\frac{1}{N} (p_{0}S)^{\frac{N}{2}}\right\}.$$
\section{Existence results for $\beta=0$}
We assume that $f(x,u)$ can be written as
\begin{equation}
f(x,u)=a(x)u+g(x,u), \label{eq1}
\end{equation}
with
\begin{equation}
a(x)\in L^{\infty}(\Omega),
\end{equation}
\begin{equation}
\begin{array}{ll}
\textrm{there exists $2<\alpha\leq 2_{*}$ such that, for every $x\in\R^{N}$ and $u\in \R$,}\\
\alpha G(x, u)\leq u\, g(x, u),\,\,\textrm{where $G(x,u)=\int_{0}^{u}g(x,t)dt$},
\label{eqI2}
\end{array}
\end{equation}
\begin{equation}
|g(x,u)|=o(|u|)\quad\textrm{as\quad $u\rightarrow 0$,\quad
uniformly in $x$,} \label{eq2}
\end{equation}
\begin{equation}
|g(x,u)|=O(|u|^{2_{*}-1})\quad\textrm{as $|u|\rightarrow +\infty$,\quad
uniformly in $x$.} \label{eq3}
\end{equation}
or
\begin{equation}
\begin{array}{ll}
|g(x,u)|=O(|u|^{r-1})\quad\textrm{as $|u|\rightarrow +\infty$,\quad
uniformly in $x$},\\ \textrm{where $r$ is such that $2_{*}<r<2^{*}$,}
\end{array}
\label{eqI1}
\end{equation}
Moreover, we assume that the first eigenvalue $\lambda_{1}(a)$ of the following problem is positive:
\begin{eqnarray*}
 \left\{
\begin{array}{lll}
 -\textsl{div}(p(x)u)-a(x)u=\mu u &\textrm{in $\Omega$}\\
  \frac{\partial u}{\partial \nu}=0
&\textrm{on $\partial\Omega$,}
\end{array}
\right.
\end{eqnarray*}
 That is,
\begin{equation}
\lambda_{1}(a)=\displaystyle\inf_{u\,\in H^{1}(\Omega)} \left\{\int_{\Omega}|\nabla u|^{2}-a(x)u^{2}dx,\,\,\,\int_{\Omega}u^{2}dx=1 \right\}> 0 .
\label{eq4}
\end{equation}
Under assumption (\ref{eq4}), it is easy to verify that $||u||=(\int_{\Omega}|\nabla u|^{2}-a(x)u^{2}dx)^{\frac{1}{2}}$ is a norm on $H^{1}(\Omega)$ equivalent to the usual norm $\|.\|_{H^{1}}$.\\
Let
$$\Phi(u)=\frac{1}{2}\int_{\Omega}\hspace{-2mm}p(x)|\nabla
u|^{2}dx-\int_{\Omega}\hspace{-2mm}F(x,u)
dx-\frac{1}{2_{*}}\int_{\partial\Omega}\hspace{-2mm}p(x)Q(x)|u|^{2_{*}}ds_{x},\,u\in
H^{1}(\Omega),$$
where $F(x,u)=\displaystyle \int_{0}^{u}f(x,t)dt$ for $x\in
\bar{\Omega},$ $u\in\R$.
Our main result in this section is
\begin{theorem}~\\
Assume (\ref{eq1})-(\ref{eq3}) and (\ref{eq4}) or (\ref{eq1})-(\ref{eq2}) and (\ref{eqI1})-(\ref{eq4}). Moreover suppose that
\begin{equation}
\begin{array}{ll}
\textrm{there exists some $v_{0}\in H^{1}$,\,$v_{0}\geq 0$ on
$\Omega$,\,$v_{0}\neq 0$ on $\partial\Omega$, such that}\\
\displaystyle \sup_{t\geq 0}\Phi(t
v_{0})<\frac{1}{2(N-1)}\frac{p(x_{0})}{\left(Q(x_{0})\right)^{N-2}}
S_{1}^{N-1}.
\end{array}
\label{important}
\end{equation}
Then problem (\ref{nouveau}) possesses a solution. \label{th1}
\end{theorem}
\begin{demo}~ \ref{th1}\\
Let $s=2_{*}$ when $f$ satisfies (\ref{eq3}) and $s=r$ when $f$ satisfies (\ref{eqI1}). By (\ref{eq2}) we have, for any $\varepsilon>0$, there is a
$\delta>0$ such that
\begin{equation*}
|g(x,u)|\leq \varepsilon |u|\quad\textrm{for a.e $x\in\Omega$,
and for all $ |u|\leq \delta$,}
\end{equation*}
thus, by (\ref{eq3}) or (\ref{eqI1}), we obtain
\begin{equation*}
|g(x,u)|\leq \varepsilon |u|+C |u|^{s-1}\quad\textrm{for a.e
$x\in\Omega$, and for all $ u\in \R$,}
\end{equation*}
and for some constant $C$ (depending on $\varepsilon$). Therefore,
we have
\begin{equation}
F(x,u)\leq \frac{1}{2} a(x)u^{2}+\frac{\varepsilon}{2}
u^{2}+\frac{C}{s}|u|^{s}\quad \textrm{for a.e $x\in\Omega$,
and for all $u\in\R$.} \label{eq6}
\end{equation}
Hence we find, for all $u\in H^{1}(\Omega)$,
$$\Phi(u)\geq \frac{1}{2}\int_{\Omega}\hspace{-2mm}p(x)|\nabla
u|^{2}dx-\frac{1}{2}
\int_{\Omega}\hspace{-2mm}a(x)|u|^{2}dx-\frac{\varepsilon}{2}
\int_{\Omega}\hspace{-2mm}|u|^{2}dx-\frac{C}{s}\int_{\Omega}\hspace{-2mm}|u|^{s}
dx-\frac{1}{2_{*}}\int_{\partial\Omega}\hspace{-2mm}p(x)Q(x)|u|^{2_{*}}ds_{x}$$
Using (\ref{eq4}) we easily see that, for
$\varepsilon>0$ small enough , there exist constants $k>0$,
$C_{1}>0$ and $C_{2}>0$ such that
\begin{eqnarray*}
\Phi(u)&\geq & k
\|u\|^{2}-C_{1}\|u\|^{s}-C_{2}\|u\|^{2_{*}}\\
&\geq &\|u\|^{2}\left(
k-C_{1}\|u\|^{s-2}-C_{2}\|u\|^{2_{*}-2}\right)\quad
\textrm{for all $u\in H^{1}$,}
\end{eqnarray*}
which implies, since $2_{*}>2$ and $s>2$, for some small $\alpha>0$ there exists $\rho>0$ such that
\begin{equation}
\Phi(u)\geq \rho,\quad \textrm{provided $\|u\|=\alpha$.}
\label{AR1}
\end{equation}
At this stage, we need some notations and some estimations. We recall $S_{1}$ defined by
$$S_{1}=\inf\left\{\int_{\R^{N}_{+}}|\nabla u|^{2}dx;\,u\in
H^{1}(\R^{N}_{+}),\,\int_{\R^{N-1}}|u|^{2_{*}}dx=1\right\}$$ the best
constant for the trace embedding $H^{1}(\R^{N}_{+})$
into $L^{q}(\partial\R^{N}_{+})$, where
$R^{N}_{+}=\{x=(x',x_{N}):\,x'\in\R^{N-1},\,x_{N}>0\}$.\\
We recall from \cite{E1} and \cite{L} that the minimizing functions of
$S_{1}$ are of the form
\begin{equation}
W(x)=\frac{\gamma_{N}}{\left[|x'|^{2}+\left(1+x_{N}\right)^{2}\right]^{\frac{N-2}{2}}},
\label{eqepsilon}
\end{equation}
 where $\gamma_{N}$ is
a positive constant depending on $N$.
 We set
$$W_{\varepsilon,
x_{0}}(x)=\varepsilon^{-\frac{N-2}{2}}\phi(x)
W(\frac{x-x_{0}}{\varepsilon}),$$ where $x_{0}\in\partial\Omega$
and $\phi$ is a radial $C^{\infty}$-function such
that\begin{equation*}\phi(x)=\left\{\begin{array}{lll}
1&\textrm{if
$|x-x_{0}|\leq\frac{R}{4}$}\\
0&\textrm{if
$|x-x_{0}|>\frac{R}{2}$}\end{array}\right.\end{equation*}
with $R>0$ is a small constant.\\
From \cite{AY} and \cite{CY} we have the following estimates
 \begin{equation}
 \int_{\Omega}p(x)|\nabla W_{\varepsilon,x_{0}}|^{2}dx=p(x_{0})A_{1}- p(x_{0})H(x_{0})\left\{\begin{array}{lll}A'_{2}\varepsilon|\log{\varepsilon}|+o(\varepsilon|\log{\varepsilon}|)&\textrm{if $N=3$}\\[\medskipamount]
 A_{2}\varepsilon+o(\varepsilon) &\textrm{if $N\geq
 4$},\end{array}\right.\label{eq18}\end{equation}
\begin{equation}
\int_{\partial\Omega}p(x)Q(x)|W_{\varepsilon,x_{0}}|^{2_{*}}ds_{x}=p(x_{0})Q(x_{0})(B_{1}-H(x_{0})B_{2}\varepsilon)+o(\varepsilon)
\label{eq19}
\end{equation}
where
$A_{1}$, $A'_{2}$, $A_{2}$, $B_{1}$ and $B_{2}$ are some positive constants defined explicitly in \cite{AY}.\\
From \cite{Y1}, for some $2<r<2^{*}$, we have
\begin{equation}
\int_{\Omega}|W_{\varepsilon,x_{0}}|^{r}dx=\left\{\begin{array}{llllll}
o(\varepsilon) &\textrm{if $N\geq 4$}\\[\medskipamount]
o(\varepsilon|\ln(\varepsilon)|) &\textrm{If $N=3$.}
\end{array}
\right.
\label{eqmars1}
\end{equation}
\\Let
us notice that
\begin{equation}
S_{1}=\frac{A_{1}}{B_{1}^{\frac{2}{2_{*}}}}\quad\textrm{and}\quad A_{2}-\frac{2}{2_{*}}\frac{A_{1}B_{2}}{B_{1}}>0. \label{eqS1}
\end{equation}
On the other hand, when $f$ satisfies (\ref{eq3}), we easily see that $\displaystyle \lim_{t\rightarrow +\infty} \Phi(tW_{\varepsilon,x_{0}})=-\infty$. Then we take $v=t_{0}W_{\varepsilon,x_{0}}$, where $t_{0}>0$ is chosen large enough so that $\|v\|>\alpha$ and $\Phi(v)\leq 0$.\\
When $f$ satisfies (\ref{eqI1}), using (\ref{eq18})-(\ref{eqmars1}), we have
\begin{eqnarray*}
\Phi(t W_{\varepsilon,x_{0}})=\displaystyle t^{2} A-t^{2_{*}} B+t^{r}\left\{\begin{array}{lll}
o(\varepsilon) &\textrm{if $N\geq 4$}\\[\medskipamount]
o(\varepsilon|\ln(\varepsilon)|) &\textrm{if $N=3$}.
\end{array}
\right.
\end{eqnarray*}
Therefore, for $\varepsilon>0$ small enough, there exists many $t_{0}>0$ such that $t_{0}^{2} A-t_{0}^{2_{*}} B <0$.
Let, again, $v=t_{0} W_{\varepsilon,x_{0}}$ for $\varepsilon$ small enough when $t_{0}$ is chosen large such that $\|v\| >\alpha$ and $\Phi(v)\leq 0$.\\
Set
 \begin{equation}
c=\inf_{\cal{P}\in \cal{A}}\max_{w\in\cal{P}}\Phi(w),\label{AR3}
 \end{equation}
where $\cal{A}$ denotes the class of continuous paths joining
 $0$ to $v$.\\
Thanks to a result of Ambrosetti and Rabinowtz \cite{AR}, see also \cite{BN}, there exists a sequence $\{u_{j}\}$ in $H^{1}(\Omega)$ such that
$$\Phi(u_{j})\rightarrow c\quad \textrm{and}\quad \Phi'(u_{j})\rightarrow 0 \,\,\textrm{in $H^{-1}(\Omega)$}.$$
Looking at (\ref{important}) we see that $c<\frac{1}{2(N-1)}\frac{p(x_{0})}{\left(Q(x_{0})\right)^{N-2}}
S_{1}^{N-1}.$
\end{demo}~\\
In order to conclude the proof of
Theorem~{\ref{th1}}, we need the following Lemma.
\begin{lemma}~\\
Let $\{u_{j}\}\subset H^{1}(\Omega)$ be a sequence satisfying
\begin{equation}
\Phi(u_{j})\rightarrow
c<\frac{p(x_{0})S_{1}^{N-1}}{2(N-1)(Q(x_{0}))^{N-2}} \label{eq7}
\end{equation}
and
\begin{equation}
\Phi'(u_{j})\rightarrow 0\quad\textrm{in}\,H^{-1}(\Omega) \label{eq8}
\end{equation}
then $\{u_{j}\}$ is relatively compact in $H^{1}(\Omega)$.
\label{lem1}
\end{lemma}
{\bf{Proof of Lemma~{\ref{lem1}}:}}\\
We start by showing that $\{u_{j}\}$ is bounded in
$H^{1}(\Omega)$.\\
Using (\ref{eq1}) and (\ref{eq4}) we see that (\ref{eq7}) and (\ref{eq8}) are equivalent to
\begin{equation}
\frac{1}{2} \|u_{j}\|^{2}-\int_{\Omega}G(x,u_{j})
dx-\frac{1}{2_{*}}\int_{\partial\Omega}p(x)Q(x)|u_{j}|^{2_{*}}ds_{x}=c+o(1),
\label{eqlm1}
\end{equation}
and
\begin{equation}
\|u_{j}\|^{2}-\int_{\Omega}g(x,u_{j})u_{j}dx-\int_{\partial\Omega}p(x)Q(x)|u_{j}|^{2_{*}}ds_{x}=<\xi_{j},u_{j}>
\label{eqlm3}
\end{equation}
with $\xi_{j}\rightarrow 0$ in $H^{-1}$.\\
Taking (\ref{eqlm1})-$\frac{1}{2}$(\ref{eqlm3}), we get
\begin{equation}
\begin{array}{ll}
\displaystyle\frac{1}{2(N-1)}\int_{\partial\Omega}p(x)Q(x)|u_{j}|^{2_{*}}ds_{x}-\int_{\Omega}
G(x,u_{j})dx+\frac{1}{2}\int_{\Omega}\hspace{-2mm}g(x,u_{j})u_{j}
dx=c+o(\|u_{j}\|).
\end{array}
\label{eq9}
\end{equation}
On the other hand, (\ref{eqlm1})-$\frac{1}{2_{*}}$(\ref{eqlm3}) yields
\begin{equation}
\begin{array}{ll}
\displaystyle\frac{1}{2(N-1)}\|u_{j}\|^{2}-\int_{\Omega}
G(x,u_{j})dx+\frac{1}{2_{*}}\int_{\Omega}\hspace{-2mm}g(x,u_{j})u_{j}
dx=c+o(\|u_{j}\|).
\end{array}
\label{eq9'}
\end{equation}
Using (\ref{eqI2}), (\ref{eq9}) and (\ref{eq9'}) follow
\begin{equation}
\begin{array}{ll}
\displaystyle\frac{1}{2(N-1)}\|u_{j}\|^{2}-(1-\frac{\alpha}{2_{*}})\int_{\Omega}
G(x,u_{j})dx\leq c+o(\|u_{j}\|)
\end{array}
\label{eqsoleil1}
\end{equation}
and
\begin{equation}
\begin{array}{ll}
\displaystyle\frac{1}{2(N-1)}\int_{\partial\Omega}p(x)Q(x)|u_{j}|^{2_{*}}ds_{x}-(1-\frac{\alpha}{2})\int_{\Omega}
G(x,u_{j})dx\leq c+o(\|u_{j}\|).
\end{array}
\label{eqsoleil2}
\end{equation}
Computing $(\frac{\alpha}{2}-1)$(\ref{eqsoleil1})+$(1-\frac{\alpha}{2_{*}})$(\ref{eqsoleil2}), we obtain
\begin{equation*}
(\frac{\alpha}{2}-1)\|u_{j}\|^{2}+(1-\frac{\alpha}{2_{*}})\int_{\partial\Omega}p(x)Q(x)|u_{j}|^{2_{*}}ds_{x}\leq c+o(\|u_{j}\|).
\end{equation*}
Therefore, since $2<\alpha\leq 2_{*}$, we obtain that $\{u_{j}\}$ is bounded in $H^{1}(\Omega)$.\\
Extract a subsequence, still denoted by $u_{j}$, such that \basn
u_{j}&\rightharpoonup& u \quad\textrm{weakly in
$H^{1}(\Omega)$},\\[\medskipamount]
u_{j}&\rightarrow& u \quad\textrm{strongly in $L^{t}(\Omega)$ for all $t<2^{*}=\frac{2N}{N-2}$},\\[\medskipamount]
u_{j}&\rightarrow& u \quad\textrm{a.e. on $\Omega$},\\[\medskipamount]
f(x,u_{j})&\rightarrow& f(x,u) \quad\textrm{strongly in
$L^{\frac{r}{r-1}}(\Omega)$},\\[\medskipamount]
u_{j}&\rightharpoonup& u \quad\textrm{weakly in
$L^{2_{*}}(\partial\Omega)$}. \easn
Passing to the limit in
(\ref{eq8}), we obtain
\begin{eqnarray*}
\left\{\begin{array}{ll} -\textsl{div}(p(x)\nabla u)=f(x,u) &\textrm {in $\Omega$}\\
\frac{\partial u}{\partial\nu}=Q(x)|u|^{2_{*}-2}u &\textrm{on $\partial\Omega$}
\end{array}
\right.
\end{eqnarray*}
We shall now verify that $u\not\equiv 0$. Indeed , suppose that
$u\equiv 0$. We claim that
$$\displaystyle \int_{\Omega}f(x,u_{j})u_{j}dx\rightarrow
0\quad\textrm{and\quad} \int_{\Omega}F(x,u_{j})dx\rightarrow
0.$$
From (\ref{eq3}) or (\ref{eqI1}), let $s=2_{*}$ if $f$ satisfies (\ref{eq3}) and $s=r$ if $f$ satisfies (\ref{eqI1}), we have
\begin{equation*}
\begin{array}{ll}
\quad\textrm{for some constants $C_{1}>0$ and $ C_{2}>0$}\\
|f(x,u)|\leq C_{1} |u|^{s-1}+C_{2}\quad \textrm{for a.e.
$x\in\Omega$, and for all $u\in \R$,}
\end{array}
\label{eq10}
\end{equation*}
and then
\begin{equation*}
|F(x,u)|\leq \frac{C_{1}}{s}|u|^{s}+C_{2} |u|\quad \textrm{for
a.e $x\in\Omega$, and for all $u\in \R$.}
\label{eq11}
\end{equation*}
Therefore
\begin{equation*}
\left|\int_{\Omega}f(x,u_{j})u_{j}dx\right|\leq
C_{1} \int_{\Omega}|u_{j}|^{s}dx+ C_{2}
\int_{\Omega}|u_{j}|dx
\end{equation*}
and
\begin{equation*}
\left|\int_{\Omega}F(x,u_{j})dx\right|\leq
\frac{C_{1}}{s} \int_{\Omega}|u_{j}|^{s}dx+ C_{2}
\int_{\Omega}|u_{j}|dx.
\end{equation*}
Since $u_{j}\rightarrow 0$ in $L^{s}(\Omega)$ then for $j$ large enough, we have
\begin{equation*}
\int_{\Omega}f(x,u_{j})u_{j}dx=o(1)
\end{equation*}
and
\begin{equation*}
\int_{\Omega}F(x,u_{j})dx=o(1).
\end{equation*}
Which gives the desired result.\\
Extracting a subsequence, still denoted by $u_{j}$, we may assume
that
\begin{equation}
\int_{\Omega}p(x)|\nabla u_{j}|^{2}dx\rightarrow l
\quad\textrm{for some constant $l\geq 0$.}
\label{eq12}
\end{equation}
Passing to the limit in (\ref{eqlm3}), we obtain
\begin{equation}
\int_{\partial\Omega}p(x)Q(x)|u_{j}|^{2_{*}}ds_{x}\rightarrow l.
\label{eq13}
\end{equation}
Passing to the limit in (\ref{eq9}), we easily get
\begin{equation}
\frac{1}{2(N-1)}l=c.
\label{eq14}
\end{equation}
Therefore $l>0$ and $\int_{\partial\Omega}p(x)Q(x)|u_{j}|^{2_{*}}ds_{x}>0$ for large $j$.\\
On the other hand, from the result of \cite[Theorem 02]{Z}, we know that there exists a constant $C(\Omega)>0$ such that for
every $w\in H^{1}(\Omega)$
\begin{equation*}
\int_{\Omega}|\nabla w|^{2}dx+C(\Omega)\int_{\Omega}|w|^{k}dx\geq
S_{1}\left(\int_{\partial\Omega}|w|^{2_{*}}ds_{x}\right)^{\frac{2}{2_{*}}},
\end{equation*}
with $k=\frac{2N}{N-1}$ if $N\geq 4$ and $k>3=\frac{2N}{N-1}$ if
$N=3$. \\
We apply this result for $w_{j}=(p(x))^{\frac{1}{2}}u_{j}$ and in
particular for $N=3$ we take $k$ such that $6=\frac{2N}{N-2}>k>3$,
we obtain for $j$ large enough
\begin{equation*}
\int_{\Omega}|\nabla
(p(x))^{\frac{1}{2}}u_{j}|^{2}dx+C(\Omega)\int_{\Omega}|(p(x))^{\frac{1}{2}}u_{j}|^{k}dx\geq
S_{1}\left(\int_{\partial\Omega}|(p(x))^{\frac{1}{2}}u_{j}|^{2_{*}}ds_{x}\right)^{\frac{2}{2_{*}}}
\end{equation*}
Since $k<\frac{2N}{N-2}$ for every $N\geq 3$, thanks to the
compact embedding $H^{1}(\Omega)\hookrightarrow L^{k}(\Omega)$, we
have, for a subsequence, $u_{j}\rightarrow 0$ strongly in $L^{k}(\Omega)$ and we
deduce
\begin{equation}
\int_{\Omega}p(x)|\nabla u_{j}|^{2}dx+o(1)\geq
S_{1}\left(\int_{\partial\Omega}|(p(x))^{\frac{1}{2}}u_{j}|^{2_{*}}ds_{x}\right)^{\frac{2}{2_{*}}}+o(1).
\label{eq15}
\end{equation}
Using the fact that
\begin{equation*}
\frac{|Q(x)|^{N-2}}{p(x)}\leq \frac{(Q(x_{0}))^{N-2}}{p(x_{0})}
\qquad \forall x\in\partial\Omega,
\end{equation*}
(\ref{eq15}) becomes
\begin{equation}
\begin{array}{lll}
\displaystyle\int_{\Omega}p(x)|\nabla u_{j}|^{2}dx+o(1)&\geq&
S_{1}\left(\displaystyle
\int_{\partial\Omega}(p(x))^{\frac{2_{*}}{2}}\left(\frac{\frac{|Q(x)|^{N-2}}{p(x)}}{\frac{(Q(x_{0}))^{N-2}}{p(x_{0})}}\right)^{\frac{1}{N-2}}|u_{j}|^{2_{*}}ds_{x}\right)^{\frac{2}{2_{*}}}+o(1)\\[\medskipamount]
&\geq & \displaystyle
S_{1}\left[\frac{\left(p(x_{0})\right)^{\frac{1}{N-2}}}{Q(x_{0})}\right]^{\frac{2}{2_{*}}}\left(\int_{\partial\Omega}p(x)|Q(x)||u_{j}|^{2_{*}}ds_{x}\right)^{\frac{2}{2_{*}}}+o(1)\\[\medskipamount]
&\geq & \displaystyle
S_{1}\left[\frac{\left(p(x_{0})\right)^{\frac{1}{N-2}}}{Q(x_{0})}\right]^{\frac{2}{2_{*}}}\left(\int_{\partial\Omega}p(x)Q(x)|u_{j}|^{2_{*}}ds_{x}\right)^{\frac{2}{2_{*}}}+o(1).
\end{array}
\label{equtile}
\end{equation}
At the limit we obtain
$$l\geq \left[\frac{\left(p(x_{0})\right)^{\frac{1}{N-2}}}{Q(x_{0})}\right]^{\frac{2}{2_{*}}} S_{1}l^{\frac{N-2}{N-1}}$$
and
$$ l\geq \frac{\left(p(x_{0})\right)^{\frac{1}{N-1}}}{\left(Q(x_{0})\right)^{\frac{N-2}{N-1}}}S_{1}l^{\frac{N-2}{N-1}}.$$
Using (\ref{AR3}) and (\ref{eq14}) we see that $l\not\equiv 0$ and
$$l^{\frac{1}{N-1}}\geq
\frac{\left(p(x_{0})\right)^{\frac{1}{N-1}}}{\left(Q(x_{0})\right)^{\frac{N-2}{N-1}}}S_{1}.$$
Therefore
$$l\geq \frac{p(x_{0})}{\left(Q(x_{0})\right)^{N-2}} S_{1}^{N-1}$$
and from (\ref{eq14}) we have
$$c\geq \frac{1}{2(N-1)}\frac{p(x_{0})}{\left(Q(x_{0})\right)^{N-2}} S_{1}^{N-1}$$
which gives a contradiction with the fact that
$c<\frac{1}{2(N-1)}\frac{p(x_{0})}{\left(Q(x_{0})\right)^{N-2}}
S_{1}^{N-1}$, thus $u\not\equiv 0$.
Now, we shall prove, for a subsequence, that $u_{j}\rightarrow u$ strongly in $H^{1}(\Omega).$\\
We start by showing that $\Phi(u)\geq 0$. Indeed, since $u$ is a solution of (\ref{nouveau}) with $\beta=0$, we have
$$\int_{\Omega}p(x)|\nabla u|^{2}dx=\int_{\Omega}f(x,u)dx+\int_{\partial\Omega}p(x)Q(x)|u|^{2_{*}}ds_{x}.$$
On the other hand
$$\Phi(u)=\frac{1}{2}\int_{\Omega}p(x)|\nabla u|^{2}dx-\frac{1}{2_{*}}\int_{\partial\Omega}p(x)Q(x)|u|^{2_{*}}ds_{x}-\int_{\Omega}F(x,u)dx.$$
Therefore, using (\ref{eqI2}), we have
$$\Phi(u)\geq \frac{1}{2(N-1)}\int_{\Omega}p(x)|\nabla u|^{2}dx+(\frac{\alpha}{2_{*}}-1)\int_{\Omega}F(x,u)dx,$$
and
$$\Phi(u)\geq \frac{1}{2(N-1)} \int_{\partial\Omega}p(x)Q(x)|u|^{2_{*}}ds_{x}+(\frac{\alpha}{2}-1)\int_{\Omega}F(x,u)dx.$$
Since $2<\alpha\leq 2_{*}$, we deduce that $\phi(u)\geq 0$.\\
We set $v_{j}=u_{j}-u$. We have
\begin{equation}
\int_{\Omega}p(x)|\nabla u_{j}|^{2}dx=\int_{\Omega}p(x)|\nabla u|^{2}dx+\int_{\Omega}p(x)|\nabla v_{j}|^{2}dx+o(1)
\label{eqlundi1}
\end{equation}
and from \cite{BL} we deduce that
\begin{equation}
\int_{\partial\Omega}p(x)Q(x)|u_{j}|^{2_{*}}ds_{x}=\int_{\partial\Omega}p(x)Q(x)|u|^{2_{*}}ds_{x}+\int_{\partial\Omega}p(x)Q(x)|v_{j}|^{2_{*}}ds_{x}+o(1).
\label{eqlundi2}
\end{equation}
Inserting (\ref{eqlundi1}) and (\ref{eqlundi2}) into (\ref{eqlm1}) and (\ref{eqlm3}) we get
\begin{equation}
\Phi(u)+\frac{1}{2}\int_{\Omega}p(x)|\nabla v_{j}|^{2}dx-\frac{1}{2_{*}}\int_{\partial\Omega}p(x)Q(x)|v_{j}|^{2_{*}}ds_{x}=c+o(1)
\label{eqlundi3}
\end{equation}
and (looking at (\ref{eq8}))
\begin{equation}
\int_{\Omega}p(x)|\nabla v_{j}|^{2}dx-\int_{\partial\Omega}p(x)Q(x)|v_{j}|^{2_{*}}ds_{x}=o(1).
\label{eqlundi4}
\end{equation}
Extracting a subsequence, still denoted by $u_{j}$, we may assume that
$$\int_{\Omega}p(x)|\nabla v_{j}|^{2}dx\rightarrow l\quad\textrm{for some constant $l\geq 0$.}$$
From (\ref{eqlundi4}) we obtain
$$\int_{\partial\Omega}p(x)Q(x)|v_{j}|^{2}ds_{x}=l.$$
Passing to the limit in (\ref{eqlundi3}), we easily see that
\begin{equation}
\frac{1}{2(N-1)}l=c-\Phi(u).
\label{eqlundi5}
\end{equation}
Using the Sobolev embedding, see (\ref{equtile}) for details, we have
\begin{equation}
l\geq \frac{(p(x_{0}))^{\frac{1}{N-2}}}{(Q(x_{0}))^{\frac{N-2}{N-1}}}S_{1} l^{\frac{N-2}{N-1}}.
\label{eqlundi6}
\end{equation}
We claim that $l=0$. Indeed, arguing by contradiction, assuming that $l\not=0$, then (\ref{eqlundi6}) gives
$$l\geq \frac{p(x_{0})}{(Q(x_{0}))^{N-2}}S_{1}^{N-1}.$$
From (\ref{eqlundi5}), we obtain
$$c-\Phi(u)\geq \frac{1}{2(N-1)} \frac{p(x_{0})}{(Q(x_{0}))^{N-2}}S_{1}^{N-1}$$
which gives a contradiction, since $c<\frac{1}{2(N-1)} \frac{(p(x_{0})}{(Q(x_{0}))^{N-2}}S_{1}^{N-1}$ and $\Phi(u)\geq 0$. Therefore $l=0$, $c=\Phi(u)$ and $u_{j}\rightarrow u$ strongly in $H^{1}(\Omega)$.
\subsection{Sufficient conditions on $f(x,u)$ which give condition (\ref{important}):}
We claim that $W_{\varepsilon,x_{0}}$ satisfies condition
(\ref{important}) for $\varepsilon>0$ sufficiently small. Indeed, we have
$$\Phi(tW_{\varepsilon,x_{0}})=\frac{1}{2}t^{2}\int_{\Omega}p(x)|\nabla
W_{\varepsilon,x_{0}}|^{2}dx-\frac{t^{2_{*}}}{2_{*}}\int_{\partial\Omega}p(x)Q(x)|W_{\varepsilon,x_{0}}|^{2_{*}}ds_{x}
-\int_{\Omega}F(x,tW_{\varepsilon,x_{0}})dx.$$
When $f$ satisfies (\ref{eq3}), we easily see that $\displaystyle\lim_{t\rightarrow
+\infty}\Phi(tW_{\varepsilon,x_{0}})=-\infty$ and for large $t_{0}>0$ we have $\Phi(t_{0}W_{\varepsilon,x_{0}})<0$.\\
When $f$ satisfies (\ref{eqI1}), using (\ref{eq18}), (\ref{eq19})  and (\ref{eqmars1}), we have
\begin{eqnarray*}
\Phi(t W_{\varepsilon,x_{0}})=\displaystyle t^{2} A-t^{2_{*}} B+t^{r}\left\{
\begin{array}{lll}
o(\varepsilon) &\textrm{if $N\geq 4$}\\[\medskipamount]
o(\varepsilon|\ln(\varepsilon)|) &\textrm{if $N=3$}.
\end{array}
\right.
\end{eqnarray*}
Therefore, for $\varepsilon>0$ small enough, we chose $t_{0}>0$ such that $t_{0}^{2} A-t_{0}^{2_{*}} B <0$ and $\Phi(t_{0} W_{\varepsilon,x_{0}})<0$.
Therefore, in both cases, $\sup_{t\in [0,\,1]}\Phi(t\, t_{0}W_{\varepsilon,x_{0}})$ is achieved at some $0\leq\tilde{t}_{\varepsilon}\leq 1$ and $\tilde{t}_{\varepsilon}$ is bounded. In the rest of this section, we note $t_{\varepsilon}=\tilde{t}_{\varepsilon} t_{0}$.\\
From now, we can suppose that $t_{\varepsilon}>0$, indeed if $t_{\varepsilon}=0$ then
$\sup_{t\geq0}\Phi(tW_{\varepsilon,x_{0}})=0$ and the condition
(\ref{important}) is satisfied.\\
Since the derivative of the function $t\rightarrow \Phi(tW_{\varepsilon,x_{0}})$ vanishes at
$t_{\varepsilon}$ we have
\begin{equation}
t_{\varepsilon}\int_{\Omega}p(x)|\nabla
W_{\varepsilon,x_{0}}|^{2}dx-t_{\varepsilon}^{2_{*}-1}\int_{\partial\Omega}p(x)Q(x)|W_{\varepsilon,x_{0}}|^{2_{*}}ds_{x}
-\int_{\Omega}f(x,t_{\varepsilon}W_{\varepsilon,x_{0}})W_{\varepsilon,x_{0}}dx=0.
\label{eq21}
\end{equation}
We claim that
\begin{equation}
\int_{\Omega}\frac{f(x,t_{\varepsilon}W_{\varepsilon,x_{0}})W_{\varepsilon,x_{0}}}{t_\varepsilon}dx\rightarrow 0\,\,
\textrm{as\,\,$\varepsilon\rightarrow 0$}. \label{eq26}
\end{equation}
Indeed, from (\ref{eq21}), we have
\begin{equation*}
\int_{\Omega}p(x)|\nabla
W_{\varepsilon,x_{0}}|^{2}dx-t_{\varepsilon}^{2_{*}-2}\int_{\partial\Omega}p(x)Q(x)|W_{\varepsilon,x_{0}}|^{2_{*}}ds_{x}
-\int_{\Omega}\frac{f(x,t_{\varepsilon}W_{\varepsilon,x_{0}})W_{\varepsilon,x_{0}}}{t_\varepsilon}dx=0.
\label{eq25}
\end{equation*}
Using (\ref{eq1})-(\ref{eq3}) or (\ref{eq1})-(\ref{eq2}) and (\ref{eqI1}), there are $C_{1}>0$ and $C_{2}>0$
such that, for a.e. $x\in\Omega$, for all $u\in \R$, $|f(x,u)|\leq C_{1} |u|^{s-1}+C_{2} |u|$ where $s=2_{*}$ if $f$ satisfies (\ref{eq3}) and $s=r$ if $f$ satisfies (\ref{eqI1}).\\
Therefore
$$\int_{\Omega}\frac{f(x,t_{\varepsilon}W_{\varepsilon,x_{0}})W_{\varepsilon,x_{0}}}{t_\varepsilon}dx\leq
C_{1} {t_{\varepsilon}}^{s-2}{\|W_{\varepsilon,x_{0}}\|}_{L^{s}}^{s}+C_{2}
\|W_{\varepsilon,x_{0}}\|_{L^2}^{2},$$
Using the fact that, as $\varepsilon\rightarrow 0$, $t_{\varepsilon}$ is
bounded, $\|W_{\varepsilon,x_{0}}\|_{L^2(\Omega)}\rightarrow 0$ and $\|W_{\varepsilon,x_{0}}\|_{L^{s}(\Omega)}\rightarrow 0$ since $s<2^{*}$, we get directly (\ref{eq26}).\\
Consequently, for $\varepsilon>0$ small enough, (\ref{eq21}) become
\begin{equation*}
t_{\varepsilon}\int_{\Omega}p(x)|\nabla
W_{\varepsilon,x_{0}}|^{2}dx-t_{\varepsilon}^{2_{*}-1}\int_{\partial\Omega}p(x)Q(x)|W_{\varepsilon,x_{0}}|^{2_{*}}ds_{x}=o(1).
\end{equation*}
Therefore
\begin{equation}
t_{\varepsilon}\leq\left(\frac{\displaystyle\int_{\Omega}p(x)|\nabla
W_{\varepsilon,x_{0}}|^{2}dx}{\displaystyle\int_{\partial\Omega}p(x)Q(x)|W_{\varepsilon,x_{0}}|^{2_{*}}ds_{x}}\right)^{\frac{1}{2_{*}-2}}+o(1).
\label{eq22}
\end{equation}
Set
$$X_{\varepsilon}=\left(\frac{\displaystyle\int_{\Omega}p(x)|\nabla
W_{\varepsilon,x_{0}}|^{2}dx}{\displaystyle\int_{\partial\Omega}p(x)Q(x)|W_{\varepsilon,x_{0}}|^{2_{*}}ds_{x}}\right)^{\frac{1}{2_{*}-2}}$$
and
$$M_{\varepsilon}=\sup_{t\in [0,\,1]}\Phi(t\,t_{0}W_{\varepsilon,x_{0}})=\Phi(t_{\varepsilon}W_{\varepsilon,x_{0}}).$$
Since the function
$t\rightarrow\frac{1}{2}t^{2}\int_{\Omega}p(x)|\nabla
W_{\varepsilon,x_{0}}|^{2}dx-\frac{t^{2_{*}}}{2_{*}}\int_{\partial\Omega}p(x)Q(x)|W_{\varepsilon,x_{0}}|^{2_{*}}ds_{x}$
is increasing on the interval $[0,X_{\varepsilon}]$ we have, by
(\ref{eq22}),
\begin{equation*}
M_{\varepsilon}\leq\frac{1}{2}X_{\varepsilon}^{2}\int_{\Omega}p(x)|\nabla
W_{\varepsilon,x_{0}}|^{2}dx
-\frac{X_{\varepsilon}^{2_{*}}}{2_{*}}\int_{\partial\Omega}p(x)Q(x)|W_{\varepsilon,x_{0}}|^{2_{*}}ds_{x}-\int_{\Omega}F(x,t_{\varepsilon}W_{\varepsilon,x_{0}})dx.
\end{equation*}
Using (\ref{eq18})-(\ref{eqS1}) and the fact that $\int_{\Omega}|W_{\varepsilon,x_{0}}|^{2}dx=o(\varepsilon)$, we obtain
\begin{equation}
\begin{array}{lll}
M_{\varepsilon}&\leq&
\frac{1}{2(N-1)}\frac{p(x_{0})}{\left(Q(x_{0})\right)^{N-2}}
S_{1}^{N-1}-\int_{\Omega}G(x,t_{\varepsilon}W_{\varepsilon,x_{0}})dx\\[\medskipamount]
&-&\left\{\begin{array}{lll}H(x_{0})\left(\frac{A_{1}}{Q(x_{0})B_{1}}\right)^{\frac{2}{2_{*}-2}}A'_{2}\varepsilon|\log{\varepsilon}|+o(\varepsilon|\log{\varepsilon}|)&\textrm{if
$N=3$,}\\[\medskipamount]
H(x_{0})\frac{p(x_{0})}{2}\left(A_{2}-\frac{2}{2_{*}}\frac{A_{1}B_{2}}{B_{1}}\right)\varepsilon+o(\varepsilon)&\textrm{if
$N\geq 4$}.
\end{array}\right.
\label{eq23}
\end{array}
\end{equation}
At this stage, we distinguish two cases:\\
{\bf{When $H(x_{0})\leq 0$.}}\\
\begin{lemma}~\\
\label{lm2}
Assume that $f(x,\,u)$ satisfies (\ref{eq1})-(\ref{eq3}) and (\ref{eq4}) or (\ref{eq1})-(\ref{eq2}) and (\ref{eqI1})-(\ref{eq4}). Suppose
that there exists some continuous function $g(.)$ such that
\begin{equation}
g(x,u)\geq g(u)\,\,\textrm{ for a.e. $x\in \Omega$ and
for all $u\in \R$} \label{eq16}
\end{equation}
and the primitive $G(u)=\int_{0}^{u}g(t)dt$ satisfies, for $N\geq
4$
\begin{equation}
\lim_{\varepsilon\rightarrow 0}\varepsilon^{\frac{N-2}{2}}\int_{\varepsilon^{\frac{1}{2}}}^{+\infty}t^{N-1}
\int_{0}^{+\infty}G\left(\frac{t^{-(N-2)}}{(1+r^{2})^{\frac{(N-2)}{2}}}
\right)r^{N-2}drdt=+\infty. \label{eq17}
\end{equation}
and for $N=3$
\begin{equation}
\displaystyle\lim_{\varepsilon\rightarrow
0}\frac{\varepsilon^{\frac{1}{2}}}{|\ln(\varepsilon)|}\int_{\varepsilon^{\frac{1}{2}}}^{+\infty}t^{2}
\int_{0}^{+\infty}G\left(\frac{t^{-1}}{(1+r^{2})^{\frac{1}{2}}}
\right)rdrdt=+\infty.
\label{eq17''}
\end{equation}
Then condition (\ref{important}) holds.
\label{lemma1}
\end{lemma}
\begin{proof}~\\
From (\ref{eq16}) and (\ref{eqepsilon}), for $\varepsilon>0$
sufficiently small, we have
\begin{equation}
\displaystyle\int_{\Omega}G(x,t_{\varepsilon}W_{\varepsilon,x_{0}})dx\geq\displaystyle
\int_{\Omega}G\Big(\frac{A\varepsilon^{\frac{N-2}{2}}}{\left[\left(\varepsilon+x_{N}\right)^{2}+|x'-x'_{0}|^{2}\right]^{\frac{N-2}{2}}}\Big)dx
\label{eq27}
\end{equation}
for some constant $A>0$.\\
Inserting (\ref{eq27}) into (\ref{eq23})
we write
\begin{equation}
\begin{array}{ll}
M{\varepsilon}&\leq
\displaystyle\frac{1}{2(N-1)}\frac{p(x_{0})}{\left(Q(x_{0})\right)^{N-2}}
S_{1}^{N-1}-\int_{\Omega}G\Big(\frac{A\varepsilon^{\frac{N-2}{2}}}{\left[\left(\varepsilon+(x_{N}-x_{0N})\right)^{2}+|x'-x'_{0}|^{2}\right]^{\frac{N-2}{2}}}\Big)dx\\[\medskipamount]
 &-\left\{\begin{array}{lll}H(x_{0})\left(\frac{A_{1}}{Q(x_{0})B_{1}}\right)^{\frac{2}{2_{*}-2}}A'_{2}\varepsilon|\log{\varepsilon}|+o(\varepsilon|\log{\varepsilon}|)&\textrm{if
$N=3$,}\\[\medskipamount]
H(x_{0})\frac{p(x_{0})}{2}\left(A_{2}-\frac{2}{2_{*}}\frac{A_{1}B_{2}}{B_{1}}\right)\varepsilon+o(\varepsilon)&\textrm{if
$N\geq 4$}.
\end{array}\right.
\end{array}
\label{eq28}
\end{equation}
Finally, we claim that
\begin{equation}
\lim_{\varepsilon\rightarrow 0}\frac{1}{\varepsilon}
\int_{\Omega}G\Big(\frac{A\varepsilon^{\frac{N-2}{2}}}{\left[\left(\varepsilon+x_{N}\right)^{2}+|x'-x'_{0}|^{2}\right]^{\frac{N-2}{2}}}\Big)dx=+\infty\,\,\textrm{If $N\geq 4$,}
\label{eq29}
\end{equation}
and
\begin{equation}
\lim_{\varepsilon\rightarrow
0}\frac{1}{\varepsilon|\log{\varepsilon}|}
\int_{\Omega}G\Big(\frac{A\varepsilon^{\frac{1}{2}}}{\left[\left(\varepsilon+x_{N}\right)^{2}+|x'-x'_{0}|^{2}\right]^{\frac{N-2}{2}}}\Big)dx=+\infty\,\,\textrm{If $N=3$}.
\label{eq30}
\end{equation}
which implies, together with (\ref{eq28}), that
$M_{\varepsilon}<\frac{1}{2(N-1)}\frac{p(x_{0})}{\left(Q(x_{0})\right)^{N-2}}
S_{1}^{N-1}$ for $\varepsilon>0$ sufficiently small.\\
{\bf{Verification of (\ref{eq29}) and (\ref{eq30})}:}\\
\begin{equation*}
\begin{array}{lllll}
\displaystyle
\int_{\Omega}G\Big(\frac{A\varepsilon^{\frac{N-2}{2}}}{\left[\left(\varepsilon+x_{N}\right)^{2}+|x'-x'_{0}|^{2}\right]^{\frac{N-2}{2}}}\Big)dx=
\displaystyle\varepsilon^{N}\int_{\R^{N}_{+}}G\Big(\frac{A\varepsilon^{-\frac{N-2}{2}}}{\left[\left(1+y_{N}\right)^{2}+|y'|^{2}\right]^{\frac{N-2}{2}}}\Big)dy+O(1)\\[\bigskipamount]
\hspace{40mm}=\displaystyle\varepsilon^{N}\omega\int_{0}^{+\infty}\int_{0}^{+\infty}(1+y_{N})^{N-1}G\Big(\frac{1}{\left(1+y_{N}\right)^{N-2}}\frac{A\varepsilon^{-\frac{N-2}{2}}}{\left[1+r^{2}\right]^{\frac{N-2}{2}}}\Big)r^{N-2}drdy_{N}
\end{array}
\end{equation*}
where $\omega$ is the area of sphere $S^{N-2}$.\\
Using the change of variable $t=\varepsilon^{\frac{1}{2}}(1+y_{N})$ we get
\begin{eqnarray*}
\begin{array}{lllll}
\displaystyle
\int_{\Omega}G\Big(\frac{A\varepsilon^{\frac{N-2}{2}}}{\left[\left(\varepsilon+x_{N}\right)^{2}+|x'-x'_{0}|^{2}\right]^{\frac{N-2}{2}}}\Big)dx=\displaystyle\varepsilon^{\frac{N}{2}}\omega\int_{\varepsilon^{\frac{1}{2}}}^{+\infty}\int_{0}^{+\infty}t^{N-1}G\Big(\frac{1}{t^{N-2}}\frac{A}{\left(1+r^{2}\right)^{\frac{N-2}{2}}}\Big)r^{N-2}drdt.
\end{array}
\end{eqnarray*}
Then (\ref{eq29}) and (\ref{eq30}) are a consequence of (\ref{eq17}) and (\ref{eq17''}).
\end{proof}
{\bf{When $H(x_{0})>0$.}}\\
\begin{lemma}~\\
Assume that $f(x,\,u)$ satisfies (\ref{eq1})-(\ref{eq3}) and (\ref{eq4}) or (\ref{eq1})-(\ref{eq2}) and (\ref{eqI1})-(\ref{eq4}). Suppose
that there exists some continuous function $g$ such that
\begin{equation}
g(x,u)\geq g(u)\,\,\textrm{ for a.e. $x\in \Omega$ and
for all $u\in \R$} \label{eq16'}
\end{equation}
and the primitive $G(u)=\int_{0}^{u}g(t)dt$ satisfies, for $N\geq
4$
\begin{equation}
\lim_{\varepsilon\rightarrow 0}\varepsilon^{\frac{N-2}{2}}\int_{\varepsilon^{\frac{1}{2}}}^{+\infty}t^{N-1}
\int_{0}^{+\infty}G\left(\frac{t^{-(N-2)}}{(1+r^{2})^{\frac{(N-2)}{2}}}
\right)r^{N-2}drdt=0. \label{eq17'}
\end{equation}
and for $N=3$
\begin{equation}
\displaystyle\lim_{\varepsilon\rightarrow
0}\frac{\varepsilon^{\frac{1}{2}}}{|\ln(\varepsilon)|}\int_{\varepsilon^{\frac{1}{2}}}^{+\infty}t^{2}
\int_{0}^{+\infty}G\left(\frac{t^{-1}}{(1+r^{2})^{\frac{1}{2}}}
\right)rdrdt=0.
\label{eq17''}
\end{equation}
Then condition (\ref{important}) holds.
\label{lemma2}
\end{lemma}
\begin{proof}~\\
The proof of this Lemma is similar to proof of Lemma~{\ref{lm2}}.
\end{proof}
Now let us give some examples for the nonlinear perturbation.
\\
{\bf{Examples of $f$:}}~ \\
If $ H(x_{0})>0$ then the two functions $g$ below  satisfy the hypothesis of  {Lemma~{\ref{lemma1}}.

\begin{enumerate}
\item[1)] $g(x,u)=g(u)=\mu\, |u|^{r-2}u$ with $\mu>0$ and $2_{*}<r<2^{*}$ .
\item[2)] \begin{equation*}
g(x,\,u)=g(u)=\left\{\begin{array}{ll}
(3+\gamma) u^{2+\gamma} \ln(u)+ |u|^{2+\gamma} \quad \textrm{if $u>1$}\\[\medskipamount]
(3+\gamma) |u|^{2+\gamma} |\ln(u)|+ |u|^{2+\gamma} \quad\textrm{if $u<1$}
\end{array}
\right.
\end{equation*}
with $0<\gamma<\frac{2}{N-2}$.
\end{enumerate}

 If $H(x_{0})<0$ then the two functions $g$ below  satisfy {Lemma~{\ref{lemma2}}.
\begin{enumerate}
\item[1)] $g(x,\,u)=g(u)=\mu\,|u|^{r-2}u$ with $\mu\in \R$ and $2<r<2_{*}$.
\item[2)] $g(x,\,u)=g(u)=\displaystyle\pm\frac{5}{2}\frac{|u|^{\frac{3}{2}}+|u|^{\frac{7}{2}}}{(1+5 |u|^{2})^{2}}$.
\end{enumerate}
\section{Existence results in presence of two critical exponents.}
We assume that $\beta=1$ and, as in the previous section, the nonlinearity $f(x,u)$ satisfies the following basic assumptions.
\begin{equation}
f(x,u)=a(x)u+g(x,u), \label{q1}
\end{equation}
with
\begin{equation}
a(x)\in L^{\infty}(\Omega),
\end{equation}
\begin{equation}
|g(x,u)|=o(|u|)\quad\textrm{as\quad $u\rightarrow 0$,\quad
uniformly in $x$,} \label{q2}
\end{equation}
\begin{equation}
|g(x,u)|=o(|u|^{2^{*}-1})\quad\textrm{as\,\, $|u|\rightarrow +\infty$,\quad
uniformly in $x$.} \label{q3}
\end{equation}
Moreover we assume that
\begin{equation}
\lambda_{1}(a)=\displaystyle\inf \left\{\int_{\Omega}|\nabla u|^{2}-a(x)u^{2}dx,\,\,\,\int_{\Omega}u^{2}dx=1 \right\}> 0 .
\label{q4}
\end{equation}
Set
$F(x,u)=\displaystyle \int_{0}^{u}f(x,t) dt$ for $x\in
\bar{\Omega},$ $u\in\R$. Let define, for $u\in H^{1}(\Omega)$,
\begin{equation}
\Phi(u)=\frac{1}{2}\int_{\Omega}\hspace{-2mm}p(x)|\nabla
u|^{2}dx-\frac{1}{2^{*}}\int_{\Omega}|u|^{2^{*}}dx-\frac{1}{2_{*}}\int_{\partial\Omega}\hspace{-2mm}p(x)Q(x)|u|^{2_{*}}ds_{x}-\int_{\Omega}\hspace{-2mm}F(x,u)
dx.
\label{eqEnergie}
\end{equation}
Our main result in this section is
\begin{theorem}~\\
Assume (\ref{q1})-(\ref{q4}) and suppose, moreover, that
\begin{equation}
\begin{array}{ll}
\textrm{there exists some $v_{0}\in H^{1}$,\,$v_{0}\geq 0$ on
$\Omega$,\,$v_{0}\neq 0$ on $\partial\Omega$, such that}\\
\displaystyle \sup_{t\geq 0}\Phi(t
v_{0})< M(S,S_{1}),\quad \textrm{where $M(S,S_{1})$ is defined in (\ref{eqnouveau valeur1})}.
\end{array}
\label{mportant}
\end{equation}
Then, problem (\ref{nouveau}) possesses a solution. \label{Th1}
\end{theorem}
{\bf{Proof of Theorem~{\ref{Th1}}.}}\\
From (\ref{q2}) we have, for any $\varepsilon>0$, there is a
$\delta>0$ such that
\begin{equation*}
|g(x,u)|\leq \varepsilon |u|\quad\textrm{for a.e $x\in\Omega$,
and for all $ |u|\leq \delta$,}
\end{equation*}
thus, by (\ref{q3}), we obtain
\begin{equation*}
|g(x,u)|\leq \varepsilon |u|+C |u|^{2^{*}-1}\quad\textrm{for a.e
$x\in\Omega$, and for all $ u\in \R$,}
\end{equation*}
and for some constant $C$ (depending on $\varepsilon$). Therefore
we have
\begin{equation}
G(x,u)\leq \frac{1}{2} a(x)u^{2}+\frac{\varepsilon}{2}
u^{2}+\frac{C}{2^{*}}|u|^{2^{*}}\quad \textrm{for a.e $x\in\Omega$,
and for all $u\in\R$.} \label{q6}
\end{equation}
Therefore, for all $u\in H^{1}(\Omega)$,
$$\Phi(u)\geq \frac{1}{2}\int_{\Omega}\hspace{-2mm}p(x)|\nabla
u|^{2}dx-\frac{1}{2}
\int_{\Omega}\hspace{-2mm}a(x)|u|^{2}dx-\frac{\varepsilon}{2}
\int_{\Omega}\hspace{-2mm}|u|^{2}dx-\frac{C}{2^{*}}\int_{\Omega}\hspace{-2mm}|u|^{2^{*}}
dx-\frac{1}{2_{*}}\int_{\partial\Omega}\hspace{-2mm}p(x)Q(x)|u|^{2_{*}}ds_{x}$$
Using (\ref{q4}) we easily see that, for
$\varepsilon>0$ small enough , there exist constants $k>0$,
$C_{1}>0$ and $C_{2}>0$ such that
\begin{eqnarray*}
\Phi(u)&\geq & k
\|u\|_{H^{1}}^{2}-C_{1}\|u\|_{H^{1}}^{2^{*}}-C_{2}\|u\|_{H^{1}}^{2_{*}}\\
&\geq &\|u\|_{H^{1}}^{2}\left(
k-C_{1}\|u\|_{H^{1}}^{2^{*}-2}-C_{2}\|u\|_{H^{1}}^{2_{*}-2}\right)\quad
\textrm{for all $u\in H^{1}$.}
\end{eqnarray*}
Which implies, since $2_{*}>2$ and $2^{*}>2$, that for some small $\alpha>0$ there exists $\rho>0$ such that
$$\Phi(u)\geq \rho,\quad \textrm{provided $\|u\|=\alpha$.}$$
On the other hand, for any $u\in H^{1}(\Omega)$, $u\not\equiv 0$ in $\bar{\Omega}$, we have by (\ref{q3}) $\lim_{t\rightarrow +\infty}\Phi(tu)=-\infty$. Thus for later purpose we take $v=t_{0} U_{\varepsilon,x_{0}}$, where $t_{0}>0$ is chosen large enough so that $v\not\in U$ and $\Phi(v)\leq 0$.\\
Set
 \begin{equation}
c=\inf_{\cal{P}\in \cal{A}}\max_{w\in\cal{P}}\Phi(w),\label{AR3}
 \end{equation}
where $\cal{A}$ denotes the class of continuous paths joining
 $0$ to $v$.\\
Looking at (\ref{mportant}) we see that $c < M(S, S_{1})$.\\
By a result of Ambrosetti and Rabinowtz \cite{AR}, see also \cite{BN}, there exists a sequence $\{u_{j}\}$ in $H^{1}(\Omega)$ satisfying
\begin{equation}
\Phi(u_{j})\rightarrow c< M(S,S_{1})
\label{q7}
\end{equation}
and
\begin{equation}
\hspace{-46mm}\Phi'(u_{j})\rightarrow 0\quad\textrm{in}\,\,H^{-1}(\Omega) \label{q8}
\end{equation}
Using (\ref{q1}) and (\ref{q4}), from (\ref{q7}) and (\ref{q8}) we write
\begin{equation}
\frac{1}{2} \|u_{j}\|^{2}-\frac{1}{2^{*}}\int_{\Omega}|u_{j}|^{2^{*}}dx-\frac{1}{2_{*}}\int_{\partial\Omega}p(x)Q(x)|u_{j}|^{2_{*}}ds_{x}-\int_{\Omega}G(x,u_{j})=c+o(1),
\label{qLm1}
\end{equation}
and
\begin{equation}
\|u_{j}\|^{2}-\int_{\Omega}|u_{j}|^{2^{*}}dx-\int_{\partial\Omega}p(x)Q(x)|u_{j}|^{2_{*}}ds_{x}-\int_{\Omega}g(x,u_{j})u_{j}dx=<\xi_{j},u_{j}>
\label{qLm3}
\end{equation}
with $\xi_{j}\rightarrow 0$ in $H^{-1}$.\\
We start by showing that $\{u_{j}\}$ is bounded in $H^{1}(\Omega)$.\\
Computing (\ref{qLm1})$-\frac{1}{2_{*}}$ (\ref{qLm3}), we obtain
\begin{equation}
\frac{1}{2(N-1)}\|u_{j}\|^{2}+\frac{N-2}{2N(N-1)}\int_{\Omega}|u_{j}|^{2^{*}}dx
-\int_{\Omega}\left[G(x,u_{j})-\frac{1}{2_{*}}g(x,u_{j})u_{j}\right]dx=c+o(1)+<\xi_{j},u_{j}>.
\label{eqj1}
\end{equation}
On the other hand, from (\ref{q3}) we have for all $\varepsilon>0$ there exists $C>0$ such that
\begin{equation}\begin{array}{lll}
|g(x,u)|\leq \varepsilon |u|^{2^{*}-1}+C\quad \textrm{for a.e $x\in \Omega$ and for all $u\in \R$,}
\label{eqj2}
\end{array}
\end{equation}
and therefore
\begin{equation}
|G(x,u)|\leq \frac{\varepsilon}{2^{*}}|u|^{2^{*}}+C u\quad \textrm{for a.e $x\in \Omega$ and for all $u\in \R$}.
\label{eqj3}
\end{equation}
We deduce from (\ref{eqj1})-(\ref{eqj3}), after using the embedding $L^{2}(\Omega)\hookrightarrow L^{1}(\Omega)$ and $H^{1}(\Omega)\hookrightarrow L^{2}(\Omega)$ that, for $\varepsilon>0$ small enough,
$$\frac{1}{2(N-1)}\|u_{j}\|^{2}+\frac{N-2}{2N(N-1)}(1+\varepsilon)
\int_{\Omega}|u_{j}|^{2^{*}}dx-C'\|u_{j}\|\leq c+o(1)$$
for some constant $C'>0$. This gives that $\{u_{j}\}$ is bounded in $H^{1}(\Omega)$, otherwise we obtain a contradiction.\\
Extract a subsequence, still denoted by $u_{j}$, such that \basn
u_{j}&\rightharpoonup& u \quad\textrm{weakly in
$H^{1}(\Omega)$},\\[\medskipamount]
u_{j}&\rightarrow& u \quad\textrm{strongly in $L^{t}(\Omega)$ for all $t<2^{*}=\frac{2N}{N-2}$},\\[\medskipamount]
u_{j}&\rightarrow& u \quad\textrm{a.e. on $\Omega$},\\[\medskipamount]
f(x,u_{j})&\rightharpoonup& f(x,u) \quad\textrm{weakly in
$L^{\frac{2^{*}}{2^{*}-1}}(\Omega)$},\\[\medskipamount]
u_{j}&\rightharpoonup& u \quad\textrm{weakly in
$L^{2_{*}}(\partial\Omega)$},\\[\medskipamount]
u_{j}&\rightharpoonup& u \quad\textrm{weakly in $L^{2^{*}}(\Omega)$}.
\easn
We shall now verify that $u\not\equiv 0$ on $\Omega$.\\
Indeed , suppose that $u\equiv 0$. We claim that
\begin{equation}
\displaystyle \int_{\Omega}f(x,u_{j})u_{j}dx\rightarrow
0\quad\textrm{and\quad} \int_{\Omega}F(x,u_{j})dx\rightarrow
0.
\label{eqj4}
\end{equation}
From (\ref{eqj2}) and (\ref{eqj3}), we have, for all $\varepsilon>0$ there exists $C>0$ such that
\begin{equation*}
\left|\int_{\Omega}f(x,u_{j})u_{j}dx\right|\leq
\varepsilon \int_{\Omega}|u_{j}|^{2^{*}}dx+ C
\int_{\Omega}|u_{j}|dx
\end{equation*}
and
\begin{equation*}
\left|\int_{\Omega}F(x,u_{j}^{+})dx\right|\leq
\frac{\varepsilon}{2^{*}} \int_{\Omega}|u_{j}|^{2^{*}}dx+ \frac{C}{2}
\int_{\Omega}|u_{j}|^{2}dx.
\end{equation*}
Since $\{u_{j}\}$ remains bounded in $L^{2^{*}}(\Omega)$ and $u_{j}\rightarrow 0$ in $L^{2}(\Omega)$ we obtain (\ref{eqj4}).\\
Now, extruding a subsequence, still denoted by $u_{j}$, we may assume that there exist some constants $l\geq 0$, $m_{1}\geq 0$ and $m_{2}\geq 0$ such that
\begin{equation}
\int_{\Omega}p(x)|\nabla u_{j}|^{2}dx\rightarrow l,\quad
\int_{\Omega}|u_{j}|^{2^{*}}dx\rightarrow m_{1},\quad\textrm{and}\quad
\int_{\partial\Omega}p(x)Q(x)|u_{j}|^{2_{*}}ds_{x}\rightarrow m_{2}.
\label{mercredi1}
\end{equation}
Passing to the limit in (\ref{qLm1}) and (\ref{qLm3}), we get
\begin{equation}
\frac{1}{2}l-\frac{1}{2^{*}} m_{1}-\frac{1}{2_{*}}m_{2}=c\quad\textrm{and}\quad
l-m_{1}-m_{2}=0.
\label{eqmercredi2}
\end{equation}
From the result of \cite[Theorem 01]{Z}, we know that there exists a constant $C(\Omega)>0$ such that for
every $w\in H^{1}(\Omega)$
\begin{equation*}
\int_{\Omega}|\nabla w|^{2}dx+C(\Omega)\int_{\Omega}|w|^{k}dx\geq
\frac{S}{2^{\frac{2}{N}}}\left(\int_{\Omega}|w|^{2_{*}}dx\right)^{\frac{2}{2^{*}}},
\end{equation*}
with $k=\frac{2N}{N-1}$ if $N\geq 4$ and $k>3=\frac{2N}{N-1}$ if
$N=3$. \\
We apply this result for $w_{j}=(p(x))^{\frac{1}{2}}u_{j}$ and in
particular for $N=3$ we take $k$ such that $6=\frac{2N}{N-2}>k>3$,
we obtain for $j$ large enough
\begin{equation*}
\int_{\Omega}|\nabla
(p(x))^{\frac{1}{2}}u_{j}|^{2}dx+C(\Omega)\int_{\Omega}|(p(x))^{\frac{1}{2}}u_{j}|^{k}dx\geq
\frac{S}{2^{\frac{2}{N}}}\left(\int_{\Omega}|(p(x))^{\frac{1}{2}}u_{j}|^{2^{*}}dx\right)^{\frac{2}{2^{*}}}
\end{equation*}
Since $k<\frac{2N}{N-2}$ for every $N\geq 3$, thanks to the
compact embedding $H^{1}(\Omega)\hookrightarrow L^{k}(\Omega)$, we
have $u_{j}\rightarrow 0$ strongly in $L^{k}(\Omega)$ and we
deduce
\begin{equation*}
\int_{\Omega}p(x)|\nabla u_{j}|^{2}dx+o(1)\geq
\frac{S}{2^{\frac{2}{N}}}\left(\int_{\Omega}|(p(x))^{\frac{1}{2}}u_{j}|^{2_{*}}dx\right)^{\frac{2}{2^{*}}}+o(1).
\label{eq15z}
\end{equation*}
Using the fact that $p(x)\geq p_{0}$ for all $x\,\in \bar{\Omega}$, we see that
\begin{equation*}
\int_{\Omega}p(x)|\nabla u_{j}|^{2}dx+o(1)\geq
\frac{p_{0}\,S}{2^{\frac{2}{N}}}\left(\int_{\Omega}|u_{j}|^{2_{*}}dx\right)^{\frac{2}{2^{*}}}+o(1).
\end{equation*}
At the limit we obtain
\begin{equation}
(m_{1})^{\frac{2}{2^{*}}}\frac{p_{0}\,S}{2^{\frac{2}{N}}}\leq l .
\label{eqmercredi3z}
\end{equation}
On the other hand, by the same way, from \cite[Theorem 02]{Z} we have (see (\ref{equtile}) for more details)
\begin{equation}
(m_{2})^{\frac{2}{2_{*}}}\left[\frac{p(x_{0})}{(Q(x_{0}))^{N-2}}\right]^{\frac{1}{N-1}}S_{1}\leq l.
\label{eqmercredi3}
\end{equation}
Combining (\ref{eqmercredi2}), (\ref{eqmercredi3z})  and (\ref{eqmercredi3}) we obtain the following
\begin{equation}
\left\{
\begin{array}{llll}
\displaystyle\frac{1}{2(N-1)}l+\frac{N-2}{2N(N-1)}m_{1}=c\\[\medskipamount]
\displaystyle\frac{1}{N}l-\frac{N-2}{2N(N-1)}m_{2}=c\\[\medskipamount]
\displaystyle m_{1}\leq \left(\frac{2^{\frac{2}{N}}\,l}{p(a) S}\right)^{\frac{2^{*}}{2}}\\[\medskipamount]
\displaystyle m_{2}\leq \left(\frac{l}{[\frac{p(x_{0})}{(Q(x_{0}))^{N-2}}]^{\frac{1}{N-1}}S_{1}}\right)^{\frac{2_{*}}{2}}.
\end{array}
\right.
\label{eqmercredi4}
\end{equation}
An easy computation yields
\begin{equation}
\frac{1}{N}l-\frac{N-2}{2(N-1)N}(\frac{l}{\frac{p(x_{0})}{(Q(x_{0}))^{N-2}}S_{1}})^{\frac{2_{*}}{2}}\leq c \leq \frac{1}{2(N-1)}l+\frac{N-2}{2(N-1)N}(\frac{2^{\frac{2}{N}}\,l}{p_{0}\,S})^{\frac{2^{*}}{2}}.
\label{independance}
\end{equation}
We can write
$$l\leq (\frac{l}{\frac{p(x_{0})}{(Q(x_{0}))^{N-2}}S_{1}})^{\frac{2_{*}}{2}}+(\frac{2^{\frac{2}{N}}\,l}{p_{0}\,S})^{\frac{2^{*}}{2}}.$$
If $l=0$ then, since $c>0$, we obtain a contradiction and we get the desired result. Now, if $l\not=0$ we reduce to the study of the following polynomial
$$\frac{1}{(2^{-\frac{2}{N}}\,p_{0}\,S)^{\frac{N}{N-2}}}t^{2}+\frac{1}{(\frac{p(x_{0})}{(Q(x_{0}))^{N-2}}S_{1})^{\frac{N-1}{N-2}}}t-1\geq 0\quad\textrm{where $t=l^{\frac{1}{N-2}}$}.$$
Which is possible if $t\geq \displaystyle\frac{2 (\frac{p(x_{0})}{(Q(x_{0}))^{N-2}}S_{1})^{\frac{N-1}{N-2}}}{1+\sqrt{1+4E'}}$ where $E'= \displaystyle\left(\frac{\frac{p(x_{0})}{(Q(x_{0}))^{N-2}}S_{1}^{N-1}}{(2^{-\frac{2}{N}}\,p_{0}\,S)^{\frac{N}{2}}}\right)^{\frac{2}{N-2}}$.\\
From the left inequality of (\ref{independance}) and the fact that $l=t^{N-2}$, we obtain $c\geq M(S, S_{1})$ which gives a contradiction with (\ref{q7}). Consequently $u\not\equiv 0$ and $u$ is a solution of (\ref{nouveau}).
\begin{remark}~\\
If we assume that
\begin{equation}
F(x,v)\leq\frac{1}{2}f(x,v) v+\frac{1}{N}|v|^{2^{*}},\quad \textrm{for all $v\in \R$ and for for a.e $x\in\Omega$}.
\label{eqmardi2}
\end{equation}
then the previous sequence $\{u_{j}\}$ is relatively compact in $H^{1}(\Omega)$.
\end{remark}
Let $\{u_{j}\}$ be the sequence defined in the proof of Theorem~{\ref{Th1}}, we recall that $u_{j}$ converge weakly to $u$ in $H^{1}(\Omega)$. We will show that $u_{j}$ converges strongly to $u$ in $H^{1}(\Omega)$.\\
Firstly, since $u$ is a solution of (\ref{nouveau}), we have
\begin{equation*}
\int_{\Omega}p(x)|\nabla u|^{2}dx=\int_{\Omega}|u|^{2^{*}}dx+\int_{\Omega}f(x,u)udx+\int_{\partial\Omega}p(x)Q(x)|u|^{2_{*}}ds_{x}.
\end{equation*}
Therefore
\begin{equation*}
\Phi(u)=\int_{\Omega}\left\{\frac{1}{N}|u|^{2^{*}}+\frac{1}{2}f(x,u)u-F(x,u)\right\}dx+\frac{1}{2(N-1)}\int_{\partial\Omega}p(x)Q(x)|u|^{2_{*}}ds_{x}.
\end{equation*}
Using (\ref{eqmardi2}) we have $\Phi(u)\geq 0$.\\
Now, we set $v_{j}=u_{j}-u$.\\
We write
\begin{equation}
\int_{\Omega}p(x)|\nabla u_{j}|^{2}dx=\int_{\Omega}p(x)|\nabla u|^{2}dx+\int_{\Omega}p(x)|\nabla v_{j}|^{2}dx+o(1)
\label{eqjeudi1}
\end{equation}
and from \cite{BL} we deduce that
\begin{equation}
\int_{\Omega}|u_{j}|^{2^{*}}dx=\int_{\Omega}|u|^{2^{*}}dx+\int_{\Omega}|v_{j}|^{2^{*}}dx+o(1),
\label{eqjeudi2}
\end{equation}
and
\begin{equation}
\int_{\partial\Omega}p(x)Q(x)|u_{j}|^{2_{*}}ds_{x}=\int_{\partial\Omega}p(x)Q(x)|u|^{2_{*}}ds_{x}+\int_{\partial\Omega}p(x)Q(x)|v_{j}|^{2_{*}}ds_{x}+o(1).
\label{eqjeudi3}
\end{equation}
Inserting (\ref{eqjeudi1}), (\ref{eqjeudi2}) and (\ref{eqjeudi3}) into (\ref{qLm1}) and (\ref{qLm3}) we get
\begin{equation}
\Phi(u)+\frac{1}{2}\int_{\Omega}p(x)|\nabla v_{j}|^{2}dx-\frac{1}{2^{*}}\int_{\Omega}|v_{j}|^{2^{*}}dx-\frac{1}{2_{*}}\int_{\partial\Omega}p(x)Q(x)|v_{j}|^{2_{*}}ds_{x}=c+o(1)
\label{eqjeudi4}
\end{equation}
and (looking at (\ref{q8}))
\begin{equation}
\int_{\Omega}p(x)|\nabla v_{j}|^{2}dx-\int_{\Omega}|v_{j}|^{2^{*}}dx-\int_{\partial\Omega}p(x)Q(x)|v_{j}|^{2_{*}}ds_{x}=o(1).
\label{eqjeudi5}
\end{equation}
Now, we assume (for a subsequence) that exists some constants $l\geq 0$, $m_{1}\geq 0$ and $m_{2}\geq 0$ such that
$$\int_{\Omega}p(x)|\nabla v_{j}|^{2}dx\rightarrow l,\quad \int_{\Omega}|v_{j}|^{2^{*}}dx\rightarrow m_{1}\quad\textrm{and}\quad \int_{\partial\Omega}p(x)Q(x)|v_{j}|^{2_{*}}ds_{x}\rightarrow m_{2}.$$
Passing to limit in (\ref{eqjeudi4}) and (\ref{eqjeudi5}), using the Sobolev embedding, a easy computation yields
\begin{equation*}
\left\{
\begin{array}{llll}
\frac{1}{2(N-1)}l+\frac{N-2}{2N(N-1)}m_{1}=c-\Phi(u)\\[\medskipamount]
\frac{1}{N}l-\frac{N-2}{2N(N-1)}m_{2}=c-\Phi(u)\\[\medskipamount]
m_{1}\leq \left(\frac{l}{p(a) S}\right)^{\frac{2^{*}}{2}}\\[\medskipamount]
m_{2}\leq \left(\frac{l}{[\frac{p(x_{0})}{(Q(x_{0}))^{N-2}}]^{\frac{1}{N-1}}S_{1}}\right)^{\frac{2_{*}}{2}}.
\end{array}
\right.
\end{equation*}
Therefore, as in end of proof of Theorem~{\ref{Th1}}, if $l\not=0$ then $c-\Phi(u)\geq M(S,S_{1})$ which is a contradiction since $c<M(S,S_{1})$ and $\Phi(u)\geq 0$. Consequently $l=0$ and then $u_{j}\rightarrow u$ strongly in $H^{1}(\Omega)$.
\subsection{Sufficient conditions on $f(x,u)$ which give condition (\ref{mportant}):}
We recall
$$S=\inf\left\{\int_{\R^{N}}|\nabla u|^{2}dx;\,u\in
H^{1}(\R^{N}),\,\int_{\R^{N}}|u|^{2^{*}}dx=1\right\}.$$
We consider, for all $\varepsilon>0$, the following functions
\begin{equation}
U_{\varepsilon,y}(x)=\left(\frac{\varepsilon}{\varepsilon^{2}+|x'-y'|^{2}+|x_{N}-y_{N}+\mu(N-2)^{-1}\varepsilon|^{2}}\right)^{\frac{N-2}{2}},
\label{eqmardi1}
\end{equation}
where $x=(x',x_{N})$, $y=(y',y_{N})$ $\in \R^{N-1}\times ]0,\,+\infty[$, $\mu\in \R$
and $\displaystyle u_{\varepsilon, x_{0}}=\xi(x)U_{\varepsilon,x_{0}}(x),$
where $\xi$ be a radial $C^{\infty}$-function such that, for a fixed positive constant $R$,
\begin{equation*}
\xi(x)=\left\{\begin{array}{ll}
1 \,\,\textrm{if $|x-x_{0}|\leq \frac{R}{4}$}\\
0\,\,\textrm{if $|x-x_{0}|> \frac{R}{2}$}
\end{array}
\right.
\end{equation*}

It is known, see \cite{CFS} and \cite{YZ}, that $U_{\varepsilon,y}$ is a solution of the following problem
\begin{equation}
\left\{\begin{array}{lll}
-\Delta u=N(N-2) u^{\frac{N+2}{N-2}}&\textrm{in $\R^{n}_{+}$}\\[\medskipamount]
u>0 &\textrm{in $\R^{N}_{+}$}\\[\medskipamount]
-\frac{\partial u}{\partial x_{N}}=\mu\,u^{\frac{N}{N-2}} &\textrm{on $\partial\R^{N}_{+}=\R^{N-1}$}.
\end{array}
\right.
\label{equation limite}
\end{equation}
We draw on estimates made in \cite[pages 17-22]{CFT}, we write
\begin{equation}
\int_{\Omega}p(x)|\nabla u_{\varepsilon,x_{0}}|^{2}dx=p(x_{0})A_{\mu}-\mu H(x_{0})p(x_{0})\left\{\begin{array}{lll}
K_{1}\varepsilon+o(\varepsilon)&\textrm{if $N\geq 4$}\\[\medskipamount]
K_{0}\varepsilon|\log{\varepsilon}|+o(\varepsilon|\log{\varepsilon}|)&\textrm{If $N=3$,}
\end{array}
\right.
\label{ah1}
\end{equation}

\begin{equation}
\int_{\Omega}|u_{\varepsilon,x_{0}}|^{2^{*}}dx=B_{\mu}-\mu H(x_{0})K_{2}\varepsilon+o(\varepsilon)\quad \textrm{for all $N\geq 3$},
\label{ah2}
\end{equation}
\begin{equation}
\int_{\partial\Omega}p(x)Q(x)|u_{\varepsilon,x_{0}}|^{2_{*}}ds_{x}=p(x_{0})Q(x_{0})C_{\mu}+\mu H(x_{0})p(x_{0})Q(x_{0}) K_{3}\varepsilon+o(\varepsilon)\quad\textrm{for all $N\geq 3$},
\label{ah3}
\end{equation}
where $A_{\mu}$, $B_{\mu}$, $C_{\mu}$ and $K_{i}>0$ for $i\in \{0,\,1,\,2,\,3\}$ are defined by
\begin{equation}
A_{\mu}=\int_{\R^{N}_{+}}|\nabla U_{\varepsilon,x_{0}}|^{2}dx=\int_{\frac{\mu}{N-2}}^{+\infty}\int_{\R^{N-1}}\frac{|x|^{2}}{(1+|x|^{2})^{N}}dx,
\label{imen1}
\end{equation}
\begin{equation}
B_{\mu}=\int_{\R^{N}_{+}}|U_{\varepsilon,x_{0}}|^{2^{*}}dx=\int_{\frac{\mu}{N-2}}^{+\infty}\int_{\R^{N-1}}\frac{1}{(1+|x|^{2})^{N}}dx,
\label{imen2}
\end{equation}
\begin{equation}
C_{\mu}=\int_{R^{N-1}}|U_{\varepsilon,x_{0}}|^{2_{*}}dx'=
\frac{1}{(1+(\frac{\mu}{N-2})^{2})^{\frac{N-2}{2}}}\int_{\R^{N-1}}\frac{1}{(1+|y|^{2})^{N-1}}dy,
\label{imen3}
\end{equation}
\begin{equation}
K_{1}=(N-2)^{2}\left(\frac{N+1}{N-3}+2\frac{N-1}{N-3}\mu^{2}\right) K_{2},
\label{imen4}
\end{equation}
\begin{equation}
K_{3}=2 (N-1)\mu K_{2}
\label{imen5}
\end{equation}
with $K_{2}>0$ and $K_{0}>0$ are some constants.
Let
$$J(u)=\frac{1}{2}p(x_{0})\int_{\R^{N}_{+}}|\nabla u|^{2}dx-\frac{1}{2^{*}}\int_{\R^{N}_{+}} |u|^{2^{*}}dx-\frac{1}{2_{*}}p(x_{0})Q(x_{0})\int_{R^{N-1}}|u(x',0)|^{2_{*}}dx',$$
We have the following result
\begin{proposition}
We have
\begin{equation*}
\inf_{u\,\in H^{1}(\R^{N}_{+})\setminus \{0\}}\max_{t\geq 0}J(t u)\leq M(S,\,S_{1}),\quad \textrm{where $M(S,\,S_{1})$ is defined in (1.4)}.
\end{equation*}
\label{pr1}
\end{proposition}
\begin{proof}~\\
We have
\begin{equation*}
J(t \,U_{\varepsilon, 0})=\frac{t^{2}}{2}p(x_{0}) A_{\mu}-\frac{t^{2^{*}}}{2^{*}}B_{\mu}-\frac{t^{2_{*}}}{2_{*}}p(x_{0})Q(x_{0}) C_{\mu}.
\end{equation*}
Set $h(t)=\frac{t^{2}}{2}p(x_{0}) A_{\mu}-\frac{t^{2^{*}}}{2^{*}}B_{\mu}-\frac{t^{2_{*}}}{2_{*}} p(x_{0})Q(x_{0}) C_{\mu}.$\\
Therefore
$$\max_{t\geq 0}J(t \,U_{\varepsilon, 0})=\max_{t\geq 0} h(t).$$
Let $t_{\mu}$ such that $h(t_{\mu})=\displaystyle\max_{t\geq 0}h(t)$. Then $t_{\mu}$ satisfies
\begin{equation}
p(x_{0}) A_{\mu}-B_{\mu}t_{\mu}^{\frac{4}{N-2}}-p(x_{0})Q(x_{0}) C_{\mu}t_{\mu}^{\frac{2}{N-2}}=0.
\end{equation}
Looking at the polynomial $B_{\mu}l^{2}+p(x_{0})Q(x_{0}) C_{\mu}l-p(x_{0}) A_{\mu}$ we deduce that
\begin{eqnarray*}
t_{\mu}&=&\left[\frac{-p(x_{0})Q(x_{0}) C_{\mu}+\sqrt{(p(x_{0})Q(x_{0}) C_{\mu})^{2}+4p(x_{0}) A_{\mu}B_{\mu}}}{2B_{\mu}} \right]^{\frac{N-2}{2}}\\[\medskipamount]
&=&2^{\frac{N-2}{2}}\left(\frac{ A_{\mu}}{Q(x_{0}) C_{\mu}}\right)^{\frac{N-2}{2}}\frac{1}{\left[1+\sqrt{1+4\frac{p(x_{0}) A_{\mu}B_{\mu}}{(Q(x_{0}) C_{\mu})^{2}}}\right]^{\frac{N-2}{2}}}.
\end{eqnarray*}
Hence
\begin{equation*}
h(t_{\mu})=t_{\mu}^{2}\left[\frac{p(x_{0})A_{\mu}}{N}-\frac{N-2}{2N(N-1)}p(x_{0})Q(x_{0})C_{\mu}t_{\mu}^{\frac{2}{N-2}}\right].
\end{equation*}
By a standard computation we have
\begin{equation*}
h(t_{\mu})= A_{\mu}\left(\frac{\left(\frac{2 A_{0}}{Q(x_{0}) C_{0}}\right)}{1+\sqrt{1+4\frac{A_{\mu}B_{\mu}}{p(x_{0})(Q(x_{0}))^{2}C_{\mu}^{2}}}} \right)^{N-2}\left[\frac{1}{N}-\frac{N-2}{N(N-1)}\frac{1}{1+\sqrt{1+4\frac{A_{\mu}B_{\mu}}{p(x_{0})(Q(x_{0}))^{2}C_{\mu}^{2}}}}\right]
\end{equation*}
From (\ref{imen1})-(\ref{imen3}) we see, for $\mu>0$ small enough, that
\begin{equation*}
\begin{array}{lll}
h(t_{\mu})&= A_{0}\left(\frac{\left(\frac{2 A_{0}}{Q(x_{0}) C_{0}}\right)}{1+\sqrt{1+4\frac{A_{0}B_{0}}{p(x_{0})(Q(x_{0}))^{2}C_{0}^{2}}}} \right)^{N-2}\left[\frac{1}{N}-\frac{N-2}{N(N-1)}\frac{1}{1+\sqrt{1+4\frac{A_{0}B_{0}}{p(x_{0})(Q(x_{0}))^{2}C_{0}^{2}}}}\right]\\[\medskipamount]
&+\mu L +o(\mu),
\end{array}
\end{equation*}
where $L$ is a constant.\\
Using the fact that $S_{1}=\frac{A_{0}}{(C_{0})^{\frac{2}{2_{*}}}}$, $S=\frac{A_{\infty}}{(B_{\infty})^{\frac{2}{2^{*}}}}$, $A_{\infty}=2 A_{0}$ and $B_{\infty}=2B_{0}$, we obtain, for $\mu > 0$ small enough, that
\begin{equation*}
\max_{t\geq 0}J(t \,U_{\varepsilon, 0})=h(t_{\mu})=M(S,\,S_{1})+\mu\,L+o(\mu).
\end{equation*}
 This gives the desired result.
\end{proof}
Now, we will show, under some additional conditions on $f(x,u)$, that $u_{\varepsilon,x_{0}}$, defined by (\ref{eqmardi1}), satisfies condition (\ref{mportant}).\\
We have
\begin{eqnarray*}
\Phi(tu_{\varepsilon,x_{0}})&=&\frac{1}{2}t^{2}\int_{\Omega}p(x)|\nabla
u_{\varepsilon,x_{0}}|^{2}dx-\frac{t^{2^{*}}}{2^{*}}\int_{\Omega}|u_{\varepsilon,x_{0}}|^{2^{*}}dx
-\frac{t^{2_{*}}}{2_{*}}\int_{\partial\Omega}p(x)Q(x)|u_{\varepsilon,x_{0}}|^{2_{*}}ds_{x}\\[\medskipamount]
&-&\int_{\Omega}F(x,tu_{\varepsilon,x_{0}})dx.
\end{eqnarray*}
Since $f(x,u)$ is a lower-order perturbation of $|u|^{2^{*}-1}$, we see that $\displaystyle\lim_{t\rightarrow +\infty}\Phi(tu_{\varepsilon,x_{0}})=-\infty$. Therefore $\displaystyle\sup_{t\geq 0}\Phi(t u_{\varepsilon,x_{0}})$ is achieved at some $t_{\varepsilon}\geq 0$ and $t_{\varepsilon}$ is bounded in $\R_{+}$.\\
From now we suppose that $t_{\varepsilon}>0$, otherwise condition (\ref{mportant}) is easily satisfied.\\
We write $t_{\varepsilon}=t_{0}+O(\varepsilon)$ when $N\geq 4$ and $t_{\varepsilon}=t_{0}+O(\varepsilon|\ln(\varepsilon)|)$ when $N=3$, using (\ref{ah1})-(\ref{ah3}) we get\\
{\bf{If $N\geq 4$:}}
\begin{equation*}
\begin{array}{lllll}
\displaystyle\Phi(t_{\varepsilon} u_{\varepsilon,x_{0}})&=\displaystyle\frac{t_{\varepsilon}^{2}}{2}p(x_{0})\int_{\R^{N}_{+}}|\nabla U_{\varepsilon,0}|^{2}dx-\frac{t_{\varepsilon}^{2^{*}}}{2}\int_{\R^{N}_{+}}|U_{\varepsilon,0}|^{2^{*}}dx\\[\medskipamount]
&\displaystyle-\frac{t_{\varepsilon}^{2_{*}}}{2}p(x_{0})Q(x_{0})\int_{R^{N-1}}|U_{\varepsilon,0}(x',0)|^{2_{*}}dx'
-\frac{t_{0}^{2}}{2}\mu H(x_{0}) K_{1} \varepsilon +\frac{t_{0}^{2^{*}}}{2^{*}}\mu H(x_{0}) K_{2} \varepsilon\\[\medskipamount]
&\displaystyle-\frac{t_{0}^{2_{*}}}{2_{*}}\mu H(x_{0}) K_{3} \varepsilon -\int_{\Omega}F(x,tu_{\varepsilon,x_{0}})dx+o(\varepsilon)\\[\medskipamount]
&\leq \displaystyle\max_{t\geq 0} J(t U_{\varepsilon,0})-\frac{t_{0}^{2}}{2}\mu H(x_{0})p(x_{0}) K_{1} \varepsilon +\frac{t_{0}^{2^{*}}}{2^{*}}\mu H(x_{0}) K_{2} \varepsilon
\\[\medskipamount]
&-\displaystyle\frac{t_{0}^{2_{*}}}{2_{*}}\mu H(x_{0}) p(x_{0})Q(x_{0}) K_{3} \varepsilon -\int_{\Omega}F(x,tu_{\varepsilon,x_{0}})dx+o(\varepsilon),
\end{array}
\end{equation*}
{\bf{If $N=3$}:}
\begin{equation*}
\begin{array}{lllll}
\displaystyle\Phi(t_{\varepsilon} u_{\varepsilon,x_{0}})&=\displaystyle\frac{t_{\varepsilon}^{2}}{2}p(x_{0})\int_{\R^{N}_{+}}|\nabla U_{\varepsilon,0}|^{2}dx-\frac{t_{\varepsilon}^{2^{*}}}{2}\int_{\R^{N}_{+}}|U_{\varepsilon,0}|^{2^{*}}dx\\[\medskipamount]
&\displaystyle-\frac{t_{\varepsilon}^{2_{*}}}{2}p(x_{0})Q(x_{0})\int_{R^{N-1}}|U_{\varepsilon,0}(x',0)|^{2_{*}}dx'-\frac{t_{0}^{2}}{2}\mu H(x_{0}) K_{0} \varepsilon |\ln(\varepsilon)|\\[\medskipamount]
&-\int_{\Omega}F(x,tu_{\varepsilon,x_{0}})dx+o(\varepsilon|\ln(\varepsilon)|)\\[\medskipamount]
&\leq \displaystyle\max_{t\geq 0} J(t U_{\varepsilon,0})-\frac{t_{0}^{2}}{2}\mu H(x_{0})p(x_{0}) K_{0} \varepsilon|\ln(\varepsilon)|
-\displaystyle\int_{\Omega}F(x,tu_{\varepsilon,x_{0}})dx+o(\varepsilon|\ln(\varepsilon)|).
\end{array}
\end{equation*}
Therefore
\begin{equation}
\begin{array}{llll}
\displaystyle\Phi(t_{\varepsilon} u_{\varepsilon,x_{0}})&\leq\displaystyle  M_{1}(S,\,S_{1})- \displaystyle\int_{\Omega}F(x,tu_{\varepsilon,x_{0}})dx+o(\mu)\\[\medskipamount]
&-\mu H(x_{0})\left\{\begin{array}{ll}
 \,[\frac{t_{0}^{2}}{2}p(x_{0}) K_{1}-\frac{t_{0}^{2^{*}}}{2^{*}}K_{2}+\frac{t_{0}^{2_{*}}}{2_{*}}p(x_{0})Q(x_{0}) K_{3}]\varepsilon+o(\varepsilon) \quad &\textrm{if $N\geq 4$}\\[\medskipamount]
\frac{t_{0}^{2}}{2}p(x_{0}) K_{0}\varepsilon|\ln(\varepsilon)|+o(\varepsilon|\ln(\varepsilon)|) \quad &\textrm{if $N=3$}.
\end{array}
\right.
\end{array}
\label{tunis1}
\end{equation}
Now, we need to give a explicit form of $t_{0}$. Since $\displaystyle\sup_{t\geq 0}\Phi(t u_{\varepsilon,x_{0}})=\displaystyle\sup_{t\geq 0}h(t)$ is achieved at $t_{\varepsilon}$ then $h'(t_{\varepsilon})=0$ and letting $\varepsilon\rightarrow 0$ we get
\begin{equation}
\displaystyle\int_{R^{N}_{+}}|\nabla U_{1,0}|^{2}dx-\frac{t_{0}^{2^{*}-2}}{p(x_{0})}\int_{\R^{N}_{+}}|U_{1,0}|^{2^{*}}dx-Q(x_{0})t_{0}^{2_{*}-2}\int_{R^{N-1}}|U_{1,0}|^{2_{*}}dx'=0.
\label{tunis2}
\end{equation}
On the other hand, since $U_{1,0}$ is a solution of (\ref{equation limite}) we see that
\begin{equation}
\displaystyle\int_{R^{N}_{+}}|\nabla U_{1,0}|^{2}dx-N(N-2)\int_{\R^{N}_{+}}|U_{1,0}|^{2^{*}}dx-\mu\int_{R^{N-1}}|U_{1,0}|^{2_{*}}dx'=0.
\label{tunis3}
\end{equation}
Combining (\ref{tunis2}) and (\ref{tunis3}) we obtain
$t_{0}=(p(x_{0})\,N(N-2))^{\frac{1}{2^{*}-2}}$.\\
Using (\ref{imen4}) and (\ref{imen5}), for $N\geq 4$, we see that
\begin{equation*}
\begin{array}{lll}
\frac{t_{0}^{2}}{2}p(x_{0}) K_{1}-\frac{t_{0}^{2^{*}}}{2^{*}}K_{2}+\frac{t_{0}^{2_{*}}}{2_{*}}p(x_{0})Q(x_{0}) K_{3}&=p(x_{0})\frac{(N-2)^{2}}{2}\left(\frac{4}{N-3}+\frac{2(N-1)}{(N-3)(N-2)^{2}}\mu^{2}\right)\\[\medskipamount]
&+2\frac{(N-1)}{(N-2)}\sqrt{N(N-2)}\sqrt{p(x_{0})}p(x_{0})Q(x_{0})\mu\\[\medskipamount]
&=K>0.
\end{array}
\end{equation*}
Combining this with (\ref{tunis1}) we obtain
\begin{equation}
\begin{array}{ll}
\displaystyle\Phi(t_{\varepsilon} u_{\varepsilon,x_{0}})\leq M_{1}(S,\,S_{1})-\int_{\Omega}F(x,tu_{\varepsilon,x_{0}})dx-\mu H(x_{0})\left\{\begin{array}{ll}
K \varepsilon+o(\varepsilon)\quad &\textrm{if $N\geq 4$}\\[\medskipamount]
K_{0}\varepsilon|\ln(\varepsilon)|+o(\varepsilon|\ln(\varepsilon)|) \quad &\textrm{if $N=3$}.
\end{array}
\right.
\end{array}
\label{sami1}
\end{equation}
Using (\ref{q1}) and the fact that $\int_{\Omega}|u_{\varepsilon,x_{0}}|^{2}dx=o(\varepsilon)$, (\ref{sami1}) becomes
\begin{equation}
\begin{array}{ll}
\displaystyle\Phi(t_{\varepsilon} u_{\varepsilon,x_{0}})\leq M_{1}(S,\,S_{1})-\int_{\Omega}G(x,tu_{\varepsilon,x_{0}})dx-\mu H(x_{0})\left\{\begin{array}{ll}
K \varepsilon+o(\varepsilon)\quad &\textrm{if $N\geq 4$}\\[\medskipamount]
K_{0}\varepsilon|\ln(\varepsilon)|+o(\varepsilon|\ln(\varepsilon)|) \quad &\textrm{if $N=3$}.
\end{array}
\right.
\end{array}
\label{sami2}
\end{equation}
where $G(x,s)=\int_{0}^{s}g(x,r)dr$.

Now, we are able to give sufficient conditions on $f$ to have the condition (\ref{mportant}):
\begin{proposition}~\\
Assume that $f(x,\,u)$ satisfies (\ref{q1})-(\ref{q4}) and that $H(x_{0})>0$. Suppose that there exists some continuous function $g(.)$ such that $g(x,\,u)\geq g(u)$ for a. e. $x\in\Omega$ and for all $u\in\R$ and the primitive $G(u)=\int_{0}^{u}g(t)dt$ satisfies :
\begin{equation}
\lim_{\varepsilon\rightarrow 0}\varepsilon^{N-1}\int_{\frac{\mu}{N-2}}^{+\infty}(1+t^{2})^{\frac{N-1}{2}}
\int_{0}^{+\infty}G\left(\frac{\varepsilon^{-\frac{N-2}{2}}}{(1+t^{2})^{\frac{N-2}{2}}(1+r^{2})^{\frac{N-2}{2}}}
\right)r^{N-2}drdt=0\quad\textrm{for $N\geq4$,} \label{mo3}
\end{equation}
and
\begin{equation}
\displaystyle\lim_{\varepsilon\rightarrow
0}\frac{\varepsilon^{2}}{|\ln(\varepsilon)|}\int_{\frac{\mu}{N-2}}^{+\infty}(1+t^{2})
\int_{0}^{+\infty}G\left(\frac{\varepsilon^{-\frac{1}{2}}}{(1+t^{2})^{\frac{1}{2}}(1+r^{2})^{\frac{1}{2}}}
\right)rdrdt=0\quad\textrm{for $N=3$}.
\label{mo4}
\end{equation}
Then condition (\ref{mportant}) holds.
\label{proposition}
\end{proposition}
\begin{proof}~\\
The proof become directly from (\ref{sami2}).
\end{proof}
{\bf{Example of $f$:}}\\
All the assumptions of Proposition~{\ref{proposition}} are satisfied for the following functions:
\begin{enumerate}
\item $g(x,\,u)=g(u)= \pm \displaystyle|u|^{r-2}u\quad$  with $2<r<2_{*}$ and $u\in\R$.
\item $g(x,\,u)=g(u)=\displaystyle\frac{u^{2_{*}-1}(2_{*}\ln(u)-1)}{(\ln(u))^{2}}\quad$  for $u>0$.
\end{enumerate}

\end{document}